\documentclass[10 pt]{amsart}
\usepackage{amsmath} 
\usepackage{amsfonts}
\usepackage{amssymb}
\usepackage{amstext}
\usepackage{amsbsy}
\usepackage{amsopn}
\usepackage{amsthm}
\usepackage{amsxtra}
\usepackage{graphicx}
\usepackage{caption}
\usepackage{subcaption}
\usepackage{color}
\usepackage{hyperref}
\usepackage{enumerate}

\newtheorem{theorem}{Theorem}[section]
\newtheorem{lemma}[theorem]{Lemma}
\newtheorem{corollary}[theorem]{Corollary}
\newtheorem{proposition}[theorem]{Proposition}

\newtheorem{claim}[theorem]{Claim}
\newtheorem*{conjecture*}{Conjecture}
\newtheorem*{claim*}{Claim}
\newtheorem*{theorem*}{Theorem}

\theoremstyle{remark}
\newtheorem{remark}[theorem]{Remark}
\theoremstyle{definition}
\newtheorem{definition}[theorem]{Definition}

\newcommand{\A}{\mathcal{A}}
\renewcommand{\S}{\mathcal{S}}
\newcommand{\T}{\mathcal{T}}
\newcommand{\ZZ}{\mathbb{Z}^2}

\newcommand{\R}{\mathbb{R}}
\newcommand{\Z}{\mathbb{Z}}

\newcommand{\N}{\mathbb{N}}
\newcommand{\rst}[1]{\ensuremath{{\mathbin\upharpoonright}%
\raise-.5ex\hbox{$#1$}}}

\newcommand{\conv}{{\rm conv}}
\newcommand{\suc}{{\rm succ}}

\title{Complexity of short rectangles and periodicity}
\author{Van Cyr}
\address{Northwestern University, Evanston, IL 60208 USA}
\email{cyr@math.northwestern.edu}
\author{Bryna Kra}
\address{Northwestern University, Evanston, IL 60208 USA}
\email{kra@math.northwestern.edu}

\subjclass[2010]{37B50 (primary), 68R15, 37B10}
\keywords{Nivat's Conjecture, nonexpansive subdynamics, block complexity, periodicity}

\thanks{The authors thank the Institute Henri Poincar\'e for hospitality while part of this work was completed.  The second author was partially supported by NSF grant $1200971$.}

\begin{document}

\begin{abstract}
The Morse-Hedlund Theorem states that  a bi-infinite sequence $\eta$ in a finite alphabet is periodic if and only if there exists $n\in\N$ such that
the block complexity function $P_\eta(n)$ satisfies $P_\eta(n)\leq n$.  In dimension two, Nivat conjectured that if there exist $n,k\in\N$ such that the $n\times k$ rectangular complexity $P_{\eta}(n,k)$ satisfies $P_{\eta}(n,k)\leq nk$, then $\eta$ is periodic.  
Sander and Tijdeman showed that this holds for $k\leq2$.  We generalize their result, showing 
that Nivat's Conjecture holds for $k\leq3$.  The method involves translating the combinatorial problem to a question about the nonexpansive subspaces of a certain $\ZZ$ dynamical system, 
and then analyzing the resulting system.
\end{abstract}

\maketitle

\section{Nivat's Conjecture for patterns of height $3$}

\subsection{Background and statement of the theorem}

The Morse-Hedlund Theorem~\cite{MH} gives a classic relation between the periodicity 
of a bi-infinite sequence taking values in a finite alphabet $\A$ and the complexity of the sequence.  
For higher dimensional sequences $\eta=\bigl(\eta(\vec n)\colon 
\vec n\in\Z^d\bigr)$ with $d\geq 1$ taking values in the finite alphabet $\A$, a 
possible generalization is the Nivat Conjecture~\cite{nivat}.  To state this precisely, we define $\eta\colon\Z^d\to\A$ to be {\em periodic} if there exists $\vec m\in\Z^d$ with $\vec m\neq \vec 0$ 
such that $\eta(\vec n+\vec m) = \eta(\vec n)$ for all $\vec n\in\Z^d$ and 
define the {\em rectangular complexity $P_\eta(n,k)$} to be the number of distinct 
$n$ by $k$ rectangular patterns that occur in $\eta$. 
Nivat conjectured that for $d=2$, if there exist $n,k\in\N$ such that 
$P_\eta(n,k)\leq nk$, then $\eta$ is periodic.  
This is a two dimensional phenomenon, as counterexamples for the corresponding 
statement in dimension $d\geq 3$ were given in~\cite{ST2}.  
There are numerous partial results, including for example~\cite{ST2, EKM, QZ} (see also 
related results in~\cite{BV, C, DR}).  In~\cite{CK} we showed that under the 
stronger hypothesis that there exist $n,k\in\N$ such that $P_\eta(n,k)\leq nk/2$, then $\eta$ is 
periodic. 

We prove that Nivat's Conjecture holds for rectangular patterns of height at most $3$: 
\begin{theorem}
\label{th:main}
Suppose $\eta\colon \ZZ\to\A$, where $\A$ denotes a finite alphabet.  Assume that there exists $n\in\N$ such that $P_{\eta}(n,3)\leq 3n$.  Then $\eta$ is periodic.
\end{theorem}

If there exists $n\in\N$ such that $P_\eta(n, 1)\leq n$, periodicity of $\eta$ follows quickly from 
the Morse-Hedlund Theorem~\cite{MH}: each row is horizontally periodic of period at most $n$ and so 
$n!$ is an upper bound for the minimal horizontal period of $\eta$.  
When there exists $n\in\N$ such that $P_\eta(n,2)\leq 2n$, periodicity of $\eta$ 
was established by Sander and Tijdeman~\cite{ST}.  
The extension to patterns of height $3$ is the main result of this article.  By the obvious symmetry, 
the analogous result holds if there exists $n\in\N$ such that $P_\eta(3,n)\leq 3n$.

\subsection{Generalized complexity functions}
To study rectangular complexity, we need to consider the complexity of more general shapes.  As introduced by Sander and Tijdeman~\cite{ST2}, if $\S\subset\Z^2$ is a finite set, we define $P_\eta(\S)$ to be the number of distinct 
patterns in $\eta$ that can fill the shape $\S$.  For example, $P_\eta(n,k) = P_\eta(R_{n,k})$, where 
$R_{n,k}$ denotes the $n$ by $k$ rectangle. 
Similar to methods introduced in~\cite{CK}, 
we find subsets of $R_{n,3}$ (the {\em generating sets}) that can be used to study periodicity.  
Using the restrictive geometry imposed by patterns of height $3$, we derive stronger properties that allow us to prove periodicity only using the complexity bound $3n$, rather than $3n/2$ as relied upon in~\cite{CK}. 

\subsection{Translation to dynamics}
As in~\cite{CK}, we translate the problem to a dynamical one.  
We define a dynamical system associated with $\eta\colon\ZZ\to\A$ in a standard way: endow $\A$ with the discrete topology, $X = \A^{\Z^2}$ with the product topology, and define the $\Z^2$-action by translations 
on $X$ by $(T^{\vec u}\eta)(\vec x):=\eta(\vec x+\vec u)$ for $\vec u \in\Z^2$.  With respect to this topology, the maps 
$T^{\vec u}\colon X\to X$ are continuous.  
Let $\mathcal{O}(\eta):=\{T^{\vec u}\eta\colon\vec u\in\Z^2\}$ denote the 
$\Z^2$-orbit of $\eta\in\A^{\Z^2}$ and set $X_{\eta}:=\overline{\mathcal{O}(\eta)}$.  
When we refer to the dynamical system $X_\eta$, we implicitly assume that this means the space $X_\eta$ endowed 
with the $\Z^2$-action by the translations $T^{\vec u}$, where $\vec u\in\Z^2$.  Note that in general $\overline{\mathcal{O}(\eta)}\setminus\mathcal{O}(\eta)$ is nonempty. 

The dynamical system $X_\eta$ 
reflects the properties of $\eta$.  An often used fact is that if $F\subset\ZZ$ is finite and $f\in X_{\eta}$, then there exists $\vec u\in\ZZ$ such that $(T^{\vec u}\eta)\rst{F}=f\rst{F}$, 
where by $\cdot\rst{F}$ we mean the restriction to the region $F$.  
So, for example, if $\eta$ satisfies some complexity bound, 
such as the existence a finite set $\S\subset\Z^2$ satisfying $P_\eta(\S)\leq N$ for some 
$N\geq 1$, then every $f\in X_\eta$ satisfies the same complexity bound.  Moreover, 
if $\eta$ is periodic with some period vector, then every $f\in X_\eta$ is also periodic with 
the same period vector.
Similarly, if $\vec u\in\Z^2$ and $F\subset\Z^2$,
there is a natural correspondence between a coloring of the form 
$(T^{-\vec u}f)\rst{F}$ and a coloring $f\rst{F+\vec u}$.  

Characterizing periodicity of
$\eta\in \A^{\Z^2}$ 
amounts to studying properties of
its orbit closure $X_\eta$.  In particular, note that $\eta$ is doubly periodic if and only if it 
has two non-commensurate period vectors, or equivalently $X_{\eta}$ is finite.

\subsection{Expansive directions}

\label{sec:intro-expansive}
Restricting a more general definition given by Boyle and Lind~\cite{BL} to a 
dynamical system $X$ with a continuous $\Z^2$-action $(T^{\vec u}\colon \vec u\in\R^2)$ on $X$, we say that a line $\ell\subset\R^2$ is an {\em expansive line} if 
there exist $r> 0$ and $\delta > 0$ such that whenever $f, g\in X$ satisfy 
$d(T^{\vec u}f, T^{\vec u} g) < \delta$ for all $\vec u\in\Z^2$ with $d(\vec u, \ell) < r$, then 
$f=g$.  Any line that is not expansive is called a {\em nonexpansive line}. 

For the system $X = \A^{\Z^2}$ with the continuous $\Z^2$-action on $X$ by translation (sometimes called the {\em full $\A$-shift}), 
it is easy to see that there are no expansive lines.  However, more interesting behavior 
arises when we restrict to $X_\eta$.  

Boyle and Lind~\cite{BL} proved a general theorem that nonexpansive lines (and, more generally, subspaces) are abundant.  In 
the context of $X_\eta$ with the continuous $\Z^2$-action on $X_\eta$ 
by translation, this theorem implies that for infinite $X_\eta$,  
there exists at least one nonexpansive line.  Rephrased in our context the Boyle and Lind 
result becomes: 
\begin{theorem}[Boyle and Lind~\cite{BL}]\label{thm:doublyperiodic}
For $\eta\colon\ZZ\to\A$,  $\eta$ is doubly periodic if and only if there are no nonexpansive lines for the $\ZZ$-action by translation on $X_\eta$.
\end{theorem}

In~\cite{CK}, we further characterized the situation with a single nonexpansive line:
\begin{theorem}[Cyr and Kra~\cite{CK}]\label{thm:singlyperiodic}
Let $\eta\in\A^{\ZZ}$.  If there exists $R_{n,k}$ such that $P_{\eta}(R_{n,k})\leq nk$ and there is a unique nonexpansive line for the $\ZZ$-action by translation on $X_{\eta}$,   then $\eta$ is periodic but not doubly periodic.
\end{theorem}

Thus Theorem~\ref{th:main} follows once we show that there can not be more than a single nonexpansive line, making its proof equivalent to showing: 
\begin{theorem}
If $\eta\colon\Z^2\to\A$ 
and there exists $R_{n,k}$ such that $P_\eta(R_{n,k})\leq nk$ for some $k\leq 3$, then 
there is at most one nonexpansive line for the dynamical system $X_\eta$.  
\end{theorem}

The proof of this result occupies the remainder of the paper.

\subsection{Conventions}
Throughout the paper, we assume that $\eta\colon\ZZ\to\A$, where $\A$ 
denotes a finite alphabet with $|\A|\geq 2$ and $X_\eta=\overline{O(\eta)}$ denotes the associated 
dynamical system, endowed with the continuous transformations $T^{\vec u}$ for 
$\vec u \in\ZZ$. 
We do not explicitly mention this hypothesis again.  However, 
each time we make an assumption on the complexity, in particular the existence of $n\in\N$ 
such that $P_\eta(R_{n,3})\leq 3n$, we make this explicit. 

\section{Generating and balanced sets}
\subsection{Generating sets}
We review some definitions from~\cite{CK}, adapted to our current problem. 

If $\S\subseteq\R^2$, we denote the convex hull of $\S$ by $\conv(\S)$.  We say $\S\subseteq\ZZ$ is {\em convex} if $S=\conv(S)\cap\ZZ$ and in this case we set $\partial S$ to be the boundary of $\conv(S)$.  A {\em boundary edge} of $\S$ is an edge of the convex polygon $\partial\S$ and a {\em boundary vertex} is a vertex of $\partial\S$.  We denote the set of boundary edges by $E(\S)$ and the set of boundary vertices by $V(\S)$.  Our convention is that if $conv(\S)$ has zero area, then $E(\S)=\emptyset$.

If the area of $conv(\S)$ is positive, we orient the boundary of $\S$ positively.  This allows us to refer to a directed line as being parallel to a boundary edge of $\S$.  We say two directed lines are {\em antiparallel} if they determine the same (undirected) line, but are endowed with opposite orientations.

If $\S\subseteq\ZZ$, then $|\S|$ denotes the number of elements of $\S$.  
We define the {\em $\S$-words of $\eta$} to be
$$
\mathcal{W}_{\eta}(\S):=\left\{(T^{\vec u}\eta)\rst{\S}\colon\vec u\in\ZZ\right\}.
$$
Following Sander and Tijdeman~\cite{ST2}, we define the {\em $\eta$-complexity} of a set $\S\subset\ZZ$ by
$$
P_{\eta}(\S):=|\mathcal{W}_{\eta}(\S)|.
$$
As in~\cite{CK}, we define the {\em $\eta$-discrepancy function $D_\eta$} on the set of nonempty, finite subsets of $\ZZ$ by
$$
D_{\eta}(\S):=P_{\eta}(\S)-|\S|.
$$

For $W\subseteq\ZZ$, by an {\em $\eta$-coloring of $W$} we mean $(T^{\vec u}\eta)\rst{W}$ for some $\vec u\in\ZZ$.  

\begin{definition}
If $\S_1\subset\S_2\subset\ZZ$ are sets and $\alpha\in X_\eta$, we say that $\alpha\rst{\S_1}$ {\em extends uniquely to an $\eta$-coloring of $\S_2$ } if for all $\beta\in X_\eta$ such that $\alpha\rst{\S_1} = \beta\rst{\S_1}$, we have that $\alpha\rst{\S_2} = \beta\rst{\S_2}$.  Otherwise, we say that the coloring $\alpha\rst{\S_1}$ {\em extends non-uniquely to an $\eta$-coloring of $\S_2$}.  
\end{definition}

\begin{definition}
If $\S\subset\ZZ$ is a finite set, then $x\in\S$ is {\em $\eta$-generated} by $\S$ if every $\eta$-coloring of $\S\setminus\{x\}$ extends uniquely to an $\eta$-coloring of $\S$.    A nonempty, finite, convex subset of $\ZZ$ for which every boundary vertex is $\eta$-generated is called an {\em $\eta$-generating set}.
\end{definition}

We note that if $\S$ is an $\eta$-generating set and $\vec v\in\ZZ$, then $\S+\vec v$ is also an $\eta$-generating set.  Similarly if $\S$ is an $\eta$-generating set and $\alpha\in X_{\eta}$, then $\S$ is 
also an $\alpha$-generating set.

\begin{lemma}\label{removegenerated}
Suppose $\S\subset\ZZ$ is finite and convex, and $x\in V(\S)$.  If $x$ is $\eta$-generated by $\S$, then $D_{\eta}(\S\setminus\{x\})=D_{\eta}(\S)+1$.  If $x$ is not $\eta$-generated by $\S$, then $D_{\eta}(\S\setminus\{x\})\leq D_{\eta}(\S)$.
\end{lemma}
\begin{proof}
If $x$ is $\eta$-generated by $\S$, then $P_{\eta}(\S\setminus\{x\})=P_{\eta}(\S)$.  Then
$$
D_{\eta}(\S\setminus\{x\})=P_{\eta}(\S\setminus\{x\})-|\S|+1=P_{\eta}(\S)-|\S|+1=D_{\eta}(\S)+1.
$$
If $x$ is not $\eta$-generated by $\S$, then $P_{\eta}(\S\setminus\{x\})<P_{\eta}(\S)$.  Thus
$$
D_{\eta}(\S\setminus\{x\})=P_{\eta}(\S\setminus\{x\})-|\S|+1<P_{\eta}(\S)-|S|+1=D_{\eta}(\S)+1.
$$
Since $D_{\eta}(\S\setminus\{x\})$ and $D_{\eta}(\S)$ are both integers, $D_{\eta}(\S\setminus\{x\})\leq D_{\eta}(\S)$.
\end{proof}

\begin{corollary}\label{amountincrease}
Suppose $\S\subset\ZZ$ is finite and convex and $p_1,\dots,p_j\in\S$ are points such that for all $1\leq i\leq j$, we have $\S\setminus\{p_1,p_2,\dots,p_i\}$ is convex.  Then $D_{\eta}(\S\setminus\{p_1,\dots,p_j\})\leq D_{\eta}(\S)+j$.
\end{corollary}

\subsection{Nonexpansivity}
We reformulate the definition of expansive, and more importantly nonexpansive, in the context of a particular configuration $\eta$.  While this is a priori weaker than Boyle and Lind's definition of expansiveness introduced in Section~\ref{sec:intro-expansive}, it is easy to check that they are equivalent in the symbolic setting: 
\begin{definition}\label{def:nonexpansive}
A line  $\ell\subset\R^2$ is a {\em nonexpansive line for $\eta$} (or just a {\em nonexpansive line} 
when $\eta$ is clear from the context) if for all $r\in\R$, there exist $f_r,g_r\in X_{\eta}$ such that $f_r\neq g_r$, but 
\begin{equation*}
\text{$f_r(\vec u)=g_r(\vec u)$ for all $\vec u\in\ZZ$ such that $d(\vec u,\ell)<r$.}
\end{equation*}
We say that $\ell$ is an {\em expansive line for $\eta$} (or just an {\em expansive line}) if it is not a nonexpansive line.

If $\ell$ is a directed line, let $H(\ell)\subset\R^2$ be the half-plane whose (positively oriented) boundary passes through the origin and is parallel to $\ell$.  We say that a directed line $\ell$ is a {\em nonexpansive direction for $\eta$} (or just a {\em nonexpansive direction} when $\eta$ is clear from the context) if there exist $f,g\in X_{\eta}$ such that $f\neq g$ but $f\rst{H(\ell)}=g\rst{H(\ell)}$.  We say $\ell$ is an {\em expansive direction for $\eta$} (or just an {\em expansive direction}) if it is not a nonexpansive 
direction for $\eta$.
\end{definition}

\begin{remark}\label{remark:translationinvariant}
Notice that the set of expansive lines (similarly expansive directions, nonexpansive lines, and nonexpansive directions) is invariant under translations in $\R^2$.
\end{remark}

We summarize properties of generating sets proved in~\cite{CK} that we use here (for completeness 
we include proofs):
\begin{proposition}[\cite{CK}, Lemmas 2.3 and 3.3]\label{parallelprop}
\label{prop:generating-sets}
Suppose there exists $n\in\N$ such that $P_{\eta}(R_{n,3})\leq3n$.  Then there exists an $\eta$-generating set $\S\subseteq R_{n,3}$ and
\begin{align}
&\text{if $\S^{\prime}\subset\S$ is nonempty and convex then $D_{\eta}(\S^{\prime})\geq D_{\eta}(\S)+1$.} \label{eq:subset}
\end{align}
Moreover, for any nonexpansive direction $\ell$, there is a boundary edge $w_{\ell}\in E(\S)$ that is parallel to $\ell$.
\end{proposition}
\begin{proof}
By assumption, $D_{\eta}(R_{n,3})\leq0$.  Let $\S\subseteq R_{n,3}$ be a convex set which is minimal (with respect to the partial ordering by inclusion) among all convex subsets of $R_{n,3}$ whose discrepancy is nonpositive.  
Since $|\A|\geq 2$, the discrepancy of a set with a single element is $|\A|-1>0$, 
and so $\S$ contains at least two elements.  In particular for any $x\in V(\S)$, the set $\S\setminus\{x\}$ is nonempty and convex.  If $x\in V(\S)$ is not $\eta$-generated by $\S$, then $D_{\eta}(\S\setminus\{x\})\leq D_{\eta}(\S)$ by Lemma~\ref{removegenerated}.  Therefore, by minimality of $\S$, if $x\in V(\S)$ then $x$ is $\eta$-generated by $\S$.  This establishes that $\S$ is an $\eta$-generating set.  Claim~\eqref{eq:subset} follows from the minimality of $\S$.

Finally, suppose $\ell$ is a directed line that is not parallel to any of the edges of $\S$.  Without loss of generality, we can assume that $\ell$ points either southwest or south.  We claim that $\ell$ is expansive for $\eta$, thereby  establishing the second part of the proposition.  

Suppose this does not hold.  Let $H\subset\R^2$ be a half-plane whose (positively oriented) boundary edge is parallel to $\ell$.  Let $\ell_0$ be the translation of $\ell$ that passes through $(0,0)$ and for all $x\in\R$, set $\ell_x:=\ell_0+(x,0)$.  Since $\ell$ is nonexpansive for $\eta$, there exist $f,g\in X_\eta$ such that $f\neq g$ but $f\rst{H}=g\rst{H}$.  Let $A:=\{\vec u\in\ZZ\colon f(\vec u)\neq g(\vec u)\}$ and set
$$
x_{\max}:=\sup\{x\in\R\colon \ell_x\cap A\neq\emptyset\}.
$$
Since $f\rst{H}=g\rst{H}$ and $\ell$ points southwest or south, we have that $x_{\max}<\infty$.   
Since $\ell$ is not parallel to any of the edges of $\S$, there is a vertex $x_{\ell}\in V(\S)$ and a half-plane whose boundary is parallel to $\ell$ such that $\S\setminus\{x_{\ell}\}$ is contained in this half-plane but $x_{\ell}$ is not.  If $\ell_{x_{\max}}\cap A\neq\emptyset$, let $\vec u_{\max}\in\ell_{x_{\max}}\cap A$.  There is a translation of $\S$ that takes $x_{\ell}$ to $u_{\max}$ and $\S\setminus\{x_{\ell}\}$ is translated to the region on which $f$ and $g$ coincide.  But this is a contradiction of the fact that 
$\S$ is $\eta$-generating, as $x_{\ell}$ is $\eta$-generated by $\S$.  If  instead $\ell_{x_{\max}}\cap A=\emptyset$ let $d$ be the distance from $x_{\ell}$ to the half-plane separating $x_{\ell}$ from $\S\setminus\{x_{\ell}\}$.  Let $\vec u\in A$ be a point such that $d(\vec u,\ell_{x_{\max}})<d/2$.  Then there is again a translation of $\S$ taking $x_{\ell}$ to $\vec u$ and $\S\setminus\{x_{\ell}\}$ is translated to the region on which $f$ and $g$ coincide.  Once again, this is
a contradiction of $x_\ell$ being $\eta$-generated.  Thus $\ell$ is an expansive direction for $\eta$, 
completing the proof.
\end{proof}

\begin{corollary}\label{generatingequation}
Suppose there exists $n\in\N$ such that $P_{\eta}(R_{n,3})\leq3n$ and $\S$ is the $\eta$-generating set constructed in Proposition~\ref{parallelprop}.  Then for any $w\in E(\S)$, we have 
$$
D_{\eta}(\S\setminus w)\geq D_{\eta}(\S)+1.
$$
\end{corollary}
\begin{proof}
If $E(\S)\neq\emptyset$, then $\conv(\S)$ has positive area (recall our convention that if $\conv(\S)$ has zero area then the edge set is empty), and so by~\eqref{eq:subset} we are done.
\end{proof}

\begin{corollary}\label{cor:rational}
Suppose there exists $n\in\N$ such that $P_{\eta}(R_{n,3})\leq3n$.  If $\ell$ is a nonexpansive direction for $\eta$, then there is a translation of $\ell$ that intersects $R_{n,3}$ in at least two places.  In particular, if $\ell$ has irrational slope, then $\ell$ is an expansive direction for $\eta$.
\end{corollary}
\begin{proof}
By Proposition~\ref{parallelprop}, there exists an $\eta$-generating set $\S\subseteq R_{n,3}$ and for any nonexpansive direction $\ell$, there is an edge $w_{\ell}\in E(\S)$ parallel to $\ell$.  The two endpoints of $w_{\ell}$ are both boundary vertices of $\S$, and so in particular are integer points in $R_{n,3}$.
\end{proof}

\begin{proposition}\label{cor:linedirection}
Suppose there exists $n\in\N$ such that $P_{\eta}(R_{n,3})\leq3n$.  If $\ell$ is a nonexpansive line for $\eta$, then at least one of the orientations on $\ell$ determines a nonexpansive direction for $\eta$.  If $\tilde{\ell}$ is an expansive line for $\eta$, then both orientations on $\tilde{\ell}$ determine expansive directions for $\eta$.
\end{proposition}
\begin{proof}
If $\ell$ is a nonexpansive line, then for all $r>0$ there exist $f_r,g_r\in X_{\eta}$ such that $f_r(\vec x)=g_r(\vec x)$ whenever $d(\vec x,\ell)<r$ but $f_r\neq g_r$.  By Corollary~\ref{cor:rational},
$\ell$ is a rational line.  By Remark~\ref{remark:translationinvariant}, 
there is no loss of generality in assuming that $\ell$ passes through the origin.  Choose $A\in SL_2(\Z)$ such that $A(\ell)$ points vertically downward.  Let $\tilde{\eta}:=\eta\circ A^{-1}$, $\tilde{f}_r:=f_r\circ A^{-1}$, and $\tilde{g}_r:=g_r\circ A^{-1}$.  
For each integer $r>0$, choose $\vec x_r=(x_r,y_r)$ such that $\tilde{f}_r(\vec x_r)\neq \tilde{g}_r(\vec x_r)$ and such that $|y_r|$ is minimal among all integer points where $\tilde{f}_r$ and $\tilde{g}_r$ differ.  By the pigeonhole principle, either $x_r>0$ infinitely often or $x_r<0$ infinitely often.  Without loss say $x_r<0$ infinitely often and pass to a subsequence $r_1<r_2<\ldots$ such that $x_{r_i}<0$ for all $i=1,2,\ldots$  Then $(T^{(x_{r_i},y_{r_i})}\tilde{f}_{r_i})(x,y)=(T^{(x_{r_i},y_{r_i})}\tilde{g}_{r_i})(x,y)$ for all $(x,y)\in\ZZ$ such that $0<x\leq r_i$, 
but $(T^{(x_{r_i},y_{r_i})}\tilde{f}_{r_i})(0,0)\neq(T^{(x_{r_i},y_{r_i})}\tilde{g}_{r_i})(0,0)$.

Since $X_{\eta}$ is compact, the sequence $\left\{\tilde{f}_{r_i}\right\}_{i=1}^{\infty}$ has an accumulation point $\tilde{f}_{\infty}$.  By passing to a subsequence, which we denote 
using the same sequence $\{r_{i}\}_{i=1}^{\infty}$,  we can assume that 
$\lim_{i\to\infty}\tilde{f}_{r_{i}}=\tilde{f}_{\infty}$.
Again by compactness, the sequence $\{\tilde{g}_{r_{i}}\}_{i=1}^{\infty}$ has an accumulation point $\tilde{g}_{\infty}$.  Again passing to a subsequence, which we continue to denote by the same sequence, we can assume that 
$\lim_{i\to\infty}\tilde{g}_{r_{i}}=\tilde{g}_{\infty}$. 
Then by construction, $\tilde{f}_{\infty}(x,y)=\tilde{g}_{\infty}(x,y)$ for all $(x,y)\in\ZZ$ such that $x>0$, but $\tilde{f}_{\infty}(0,0)\neq\tilde{g}_\infty(0,0)$.  Thus the vertical direction with downward orientation is a nonexpansive direction for $\tilde{\eta}$.  Therefore, the orientation on $\ell$ inherited from the downward orientation on $A(\ell)$ is nonexpansive for $\eta$.

Since half-planes contain arbitrarily wide strips, the second part of the proposition is immediate. 
\end{proof}

The corollary shows that if $\ell$ is a nonexpansive line for $\eta$, then there is an orientation on $\ell$ that determines a nonexpansive direction for $\eta$.  We do not know, a priori, that both orientations on $\ell$ determine nonexpansive directions for $\eta$.  In the sequel, this is a significant hurdle: we 
put considerable effort into the construction of particular sets (Proposition~\ref{lemma:balanced}) which can be used to show (Proposition~\ref{cor:antiparallel}) that when there exists $n\in\N$ such that $P_{\eta}(R_{n,3})\leq 3n$, it is indeed the case that {\em both} orientations of a nonexpansive line for $\eta$ determine nonexpansive directions.

\begin{corollary}\label{cor:bound}
Suppose there exists a finite, convex set $\S\subset\ZZ$ and an edge $w\in E(\S)$ such that
$$
D_{\eta}(\S\setminus w)>D_{\eta}(\S).
$$
Then for any $w\in E(\S)$, there are at most $|w\cap\S|-1$ $\eta$-colorings of $\S\setminus w$ that do not extend uniquely to an $\eta$-coloring of $\S$.
\end{corollary}
\begin{proof}
Since $|\S\setminus w|=|\S|-|w\cap\S|$,
$$
P_{\eta}(\S\setminus w)-|\S|+|w\cap\S|=D_{\eta}(\S\setminus w)>D_{\eta}(\S)=P_{\eta}(\S)-|\S|.
$$
Therefore $P_{\eta}(\S)\leq P_{\eta}(\S\setminus w)+|w\cap\S|-1$.  On the other hand, defining $\pi\colon \mathcal{W}_{\eta}(\S)\to\mathcal{W}_{\eta}(\S\setminus w)$ to be the natural restriction,  the number of $\eta$-colorings of $\S\setminus w$ that extend non-uniquely to an $\eta$-coloring of $\S$ is the number of points in $\mathcal{W}_{\eta}(\S\setminus w)$ whose preimage under $\pi$ contains more than one element.  Since $\pi$ is surjective, this is at most $|\mathcal{W}_{\eta}(\S)|-|\mathcal{W}_{\eta}(\S\setminus w)|$.  In other words, it is at most $P_{\eta}(\S)-P_{\eta}(\S\setminus w)$.
\end{proof}

\begin{corollary}\label{modifiedgeneratingcor}
Suppose there exists $n\in\N$ such that $P_{\eta}(R_{n,3})\leq3n$.  If $\ell$ is a nonexpansive direction for $\eta$, $\T\subset\ZZ$ is a finite set, and $x\in V(\T)$ is $\eta$-generated by $\T$, then there is no translation of $\ell$ that separates $x$ from $\conv(\T\setminus\{x\})$.
\end{corollary}
\begin{proof}
The argument is a straightforward modification of the proof of~\eqref{eq:subset} in Proposition~\ref{parallelprop}.
\end{proof}

\subsection{Balanced sets}

We define the types of sets that are used to show that under the complexity assumption, 
both orientations of a nonexpansive line for $\eta$ determine nonexpansive directions:  
\begin{definition}
\label{def:balanced}
Suppose $\ell$ is a directed line.  A finite, convex 
set $\S\subset\ZZ$ is {\em $\ell$-balanced} if
\begin{enumerate}
\item There is an edge $w\in E(\S)$ parallel to $\ell$;
\item Both endpoints of $w$ are $\eta$-generated by $\S$;
\item The set $\S$ satisfies $D_{\eta}(\S\setminus w)>D_{\eta}(\S)$;
\item Every line parallel to $\ell$ that has nonempty intersection with 
$\S$ intersects $\S$ in at least $|w\cap\S|-1$ integer points.
\end{enumerate}
\end{definition}

Note that an $\ell$-balanced set is not necessarily an $\eta$-generating set.

Definition~\ref{def:balanced} is slightly less general than the 
definition of an $\ell$-balanced set used in~\cite{CK}, where an 
$\ell$-balanced does not necessarily satisfy the first condition.

The main result of this section is Proposition~\ref{prop:period}, where 
we use balanced sets to deduce the periodicity 
of certain elements of $X_{\eta}$.  In~\cite{CK}, we relied on the 
stronger assumption that $P_{\eta}(R_{n,k})\leq\frac{nk}{2}$ 
to  show the existence of balanced sets (as well as other uses related 
to the existence of generating sets with further properties).  Due to 
the simplified geometry available in rectangles of height $3$, we are 
able to avoid the stronger assumption.

We start by showing the existence of balanced sets:
\begin{proposition}
\label{lemma:balanced}
Suppose there exists $n\in\N$ such that
$P_{\eta}(R_{n,3})\leq3n$ and suppose that $\ell\subset\R^2$ is a
nonexpansive direction for $\eta$.  If $\eta$ is aperiodic, then there
exists an $\ell$-balanced subset.
\end{proposition}

\begin{proof}
Suppose $\ell$ is a nonexpansive direction for $\eta$.  We make some
simplifying assumptions.  First,
if $n=1$ then by the Morse-Hedlund Theorem~\cite{MH},  $\eta$ is periodic
and so we can assume that $n>1$.  Second, if $P_{\eta}(R_{n,2})\leq2n$,
then by Sander and
Tijdeman's Theorem~\cite{ST}, $\eta$ is periodic and so we can assume
that $P_{\eta}(R_{n,2})>2n$, meaning that 
\begin{equation}
\label{eq:discrep}
D_\eta(R_{n,3})\leq 0  < D_\eta(R_{n,2}).
\end{equation}
Finally, we can assume that
$P_{\eta}(R_{(n-1),3})>3n-3$, meaning that $n$ is chosen to be the
minimal integer satisfying $P_{\eta}(R_{n,3})\leq3n$.

We consider three cases depending on the direction of $\ell$: vertical,
horizontal, and neither vertical nor horizontal.

 By Proposition~\ref{parallelprop}, there exists  an $\eta$-generating set $\S\subset R_{n,3}$ 
 and  there is an edge $w\in E(\S)$ parallel to
$\ell$.  If $|w\cap\S|=2$, then $\S$ is $\ell$-balanced and
we are done.  Thus it suffices to assume that
$|w\cap\S|\geq3$.

\subsubsection*{Assume $\ell$ is vertical}   Suppose that $\ell$
points downward
(the case that $\ell$ points upward is similar).  Then since a vertical
line cannot intersect a subset of $R_{n,3}$ in more than three places,
$|w\cap\S|=3$.  Observe that $(0,0)$ and $(0,2)$ are both
$\eta$-generated by $R_{n,3}$ since $\S$ can be translated into 
$R_{n,3}$ in such a way that $w$ is translated to the set
$\{(0,0),(0,1),(0,2)\}$.  In this case $R_{n,3}$ is $\ell$-balanced.

\subsubsection*{Assume $\ell$ is horizontal} Suppose that $\ell$
points left
(the case that $\ell$ points right is similar).  For $0\leq a\leq b\leq
n$, set
$$
\S_{[a,b]}:=R_{n,2}\cup\left\{(x,2)\colon a\leq x\leq b\right\}.
$$
Let $\tilde{\S}$ be a minimal set of this
form (with respect
to the partial ordering by inclusion)
satisfying $D_{\eta}(\tilde{\S})\leq D_{\eta}(R_{n,3})$; say $\tilde{\S}=\S_{[a_0,b_0]}$
for some $a_0\leq b_0$.  
Suppose first
that $a_0=b_0$.  If $(a_0,2)$ is $\eta$-generated by $R_{n,2}$,
Corollary~\ref{modifiedgeneratingcor}
contradicts the fact that the horizontal is a nonexpansive direction
for $\eta$.  If  $(a_0,2)$ is not $\eta$-generated by $R_{n,2}$, then
$D_{\eta}(R_{n,2})\leq D_{\eta}(\tilde{\S})\leq D_{\eta}(R_{n,3})$, also
a contradiction of~\eqref{eq:discrep}.  Therefore we can assume $a_0<b_0$ and
$D_{\eta}(\tilde{\S})\leq D_{\eta}(R_{n,3})\leq D_{\eta}(R_{n,2})$.  By
minimality and Lemma~\ref{removegenerated}, the points $(a_0,2)$ and
$(b_0,2)$ must both be $\eta$-generated by $\tilde{\S}$.  In this case
$\tilde{\S}$ is an $\ell$-balanced set.

\subsubsection*{Assume $\ell$ is neither vertical nor horizontal}
Making a coordinate change of the form $(x,y)\mapsto(\pm
x,\pm y)$ if necessary, we can assume that $\ell$ points southwest. A
line parallel to $\ell$ cannot intersect $R_{n,3}$ in more than three
places and so $|w\cap\S|=3$.   Since $\ell$ is not
horizontal, $w\cap\S$ can have at most one integer point at any
$y$-coordinate and thus $w\cap\S$ has exactly one integer point at each of
the three $y$-coordinates in $R_{n,3}$.  Therefore there exists an
integer $a>0$ such that $(-a,-1)$ is parallel to $\ell$.
Since a translation of any $\eta$-generating set is also
$\eta$-generating, without loss of generality we can assume the
bottom-most integer point on $w$ is $(0,0)$.

We claim that any $\eta$-coloring of $R_{n,3}$ extends uniquely to an
$\eta$-coloring of the set $R_{n,3}\cup\{(-1,0),(-2,0),\dots,(-a,0)\}$.
Set $T_0:=R_{n,3}$ and for $0<i\leq a$,  define
$$
T_i:=R_{n,3}\cup\{(-1,0),(-2,0),\dots,(-i,0)\}.
$$
Then the set $\S-(i,0)$ is contained in $T_i$ and
$(\S\setminus\{(0,0)\})-(i,0)$ is contained in $T_{i-1}$.  Since
$\S-(i,0)$ is an $\eta$-generating set, the color of vertex $(-i,0)$ can
be deduced from the coloring of $\S-(i,0)$.  Thus for $0<i\leq a$, every
$\eta$-coloring of $T_{i-1}$ extends uniquely to an $\eta$-coloring of
$T_i$.  Inductively, every $\eta$-coloring of $R_{n,3}$ extends uniquely
to an $\eta$-coloring of $T_a$ and the claim follows (see Figure~\ref{fig:lemma}).

\begin{figure}
\captionsetup[subfigure]{labelformat=empty}
\begin{subfigure}[t]{0.45\textwidth}
		\centering
		  \def\svgwidth{\columnwidth}
        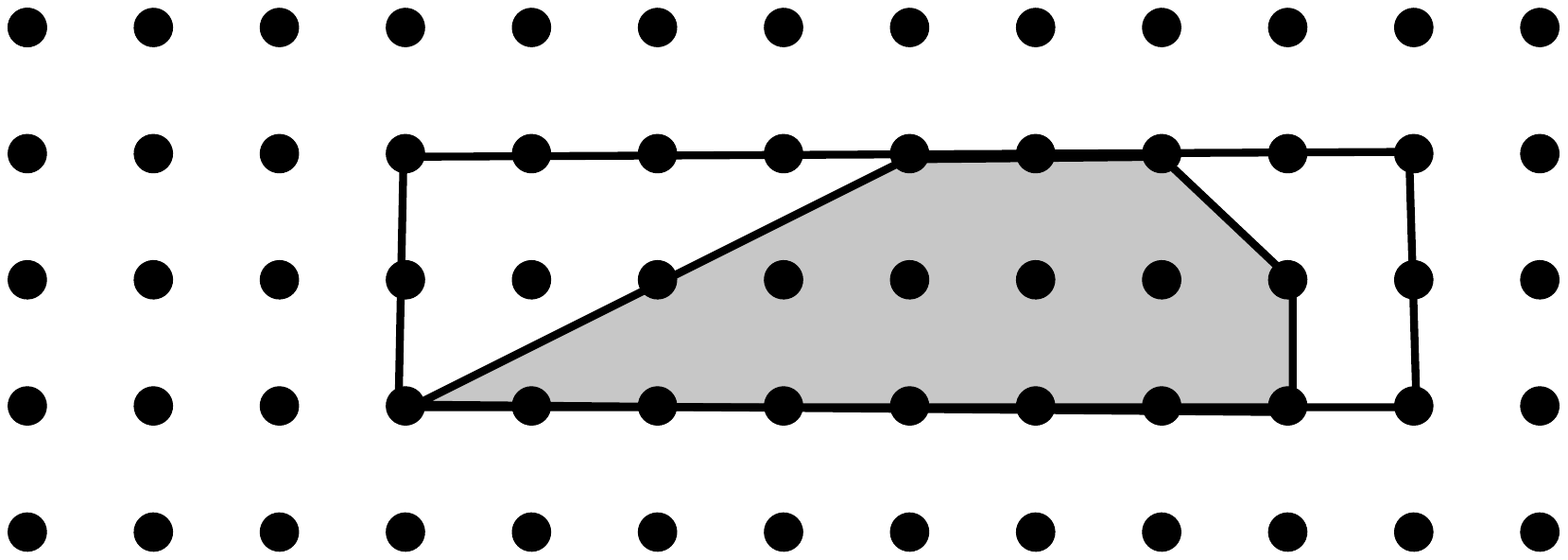
                \setlength{\abovecaptionskip}{-63mm}
		\caption{The $\eta$-generating set $\S$ is shaded.}
	\end{subfigure}
	\hspace{0.2 in}
	\begin{subfigure}[t]{0.45\textwidth}
		\centering
		  \def\svgwidth{\columnwidth}
        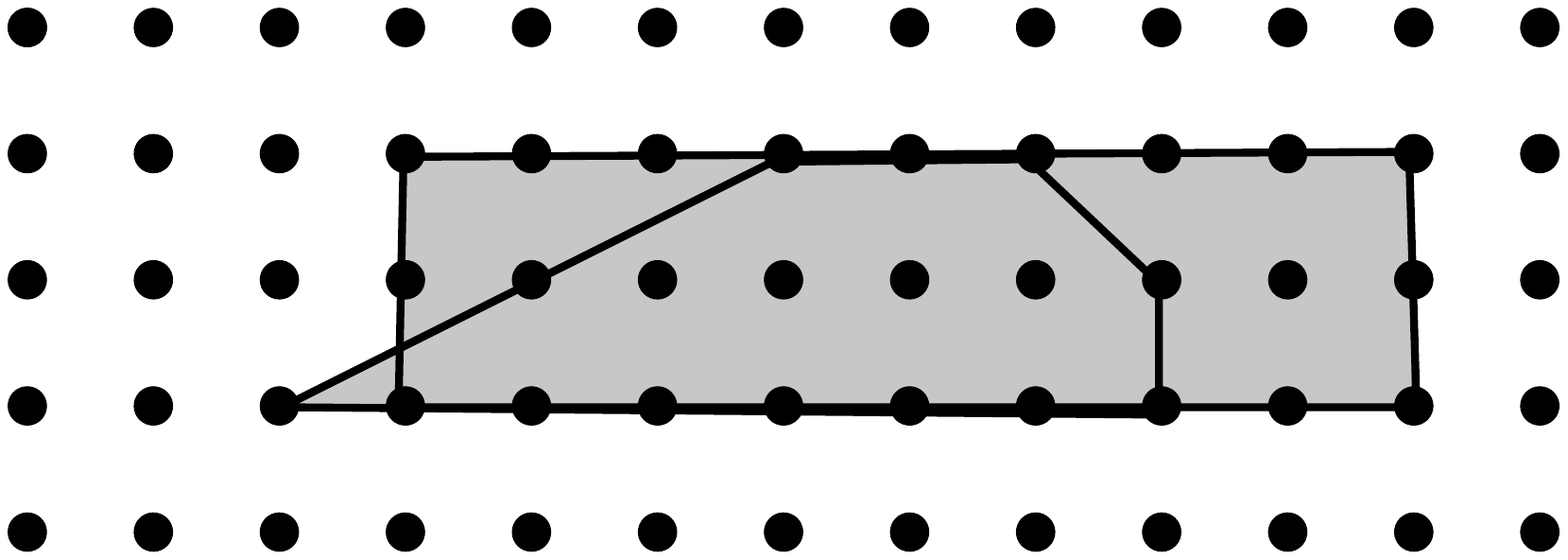
                \setlength{\abovecaptionskip}{-63mm}
		\caption{The shaded region $T_1$ contains $\S-(1,0)$.}
	\end{subfigure}
	\vspace{-2 in}
	\hspace{0.2 in}
	\begin{subfigure}[t]{0.45\textwidth}
		\centering
		  \def\svgwidth{\columnwidth}
        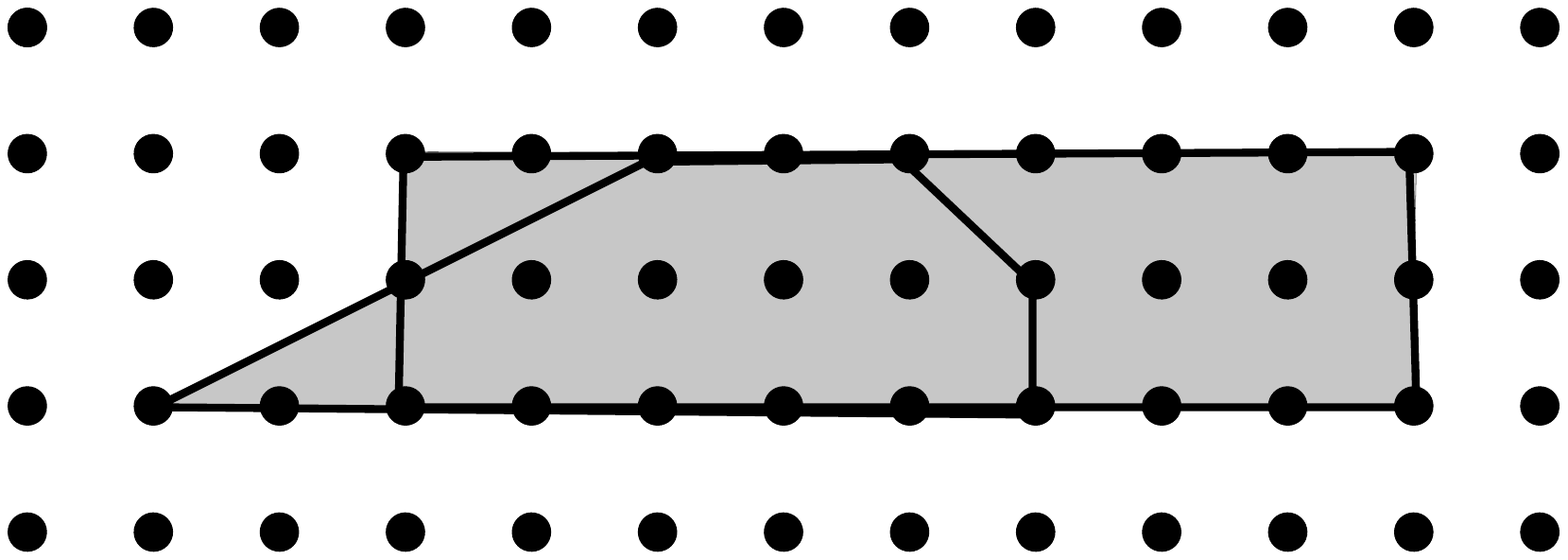
                \setlength{\abovecaptionskip}{-63mm}
		\caption{The shaded region $T_2$ contains $\S-(2,0)$.}
	\end{subfigure}
	\hspace{0.2 in}
	\begin{subfigure}[t]{0.45\textwidth}
		\centering
  \def\svgwidth{\columnwidth}
        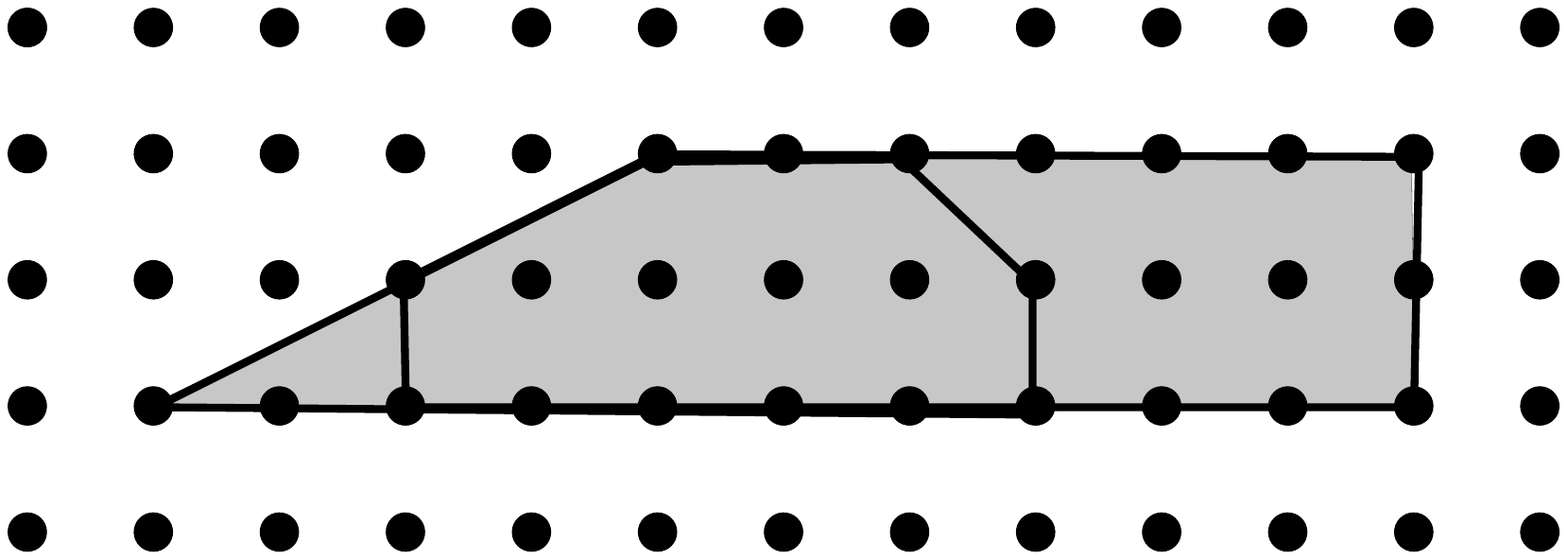
                \setlength{\abovecaptionskip}{-63mm}
                		\caption{Points not $\eta$-generated are removed.}
	\end{subfigure}
	\vspace{-2 in}
		\hspace{0.2 in}
	
	\begin{subfigure}[t]{0.45\textwidth}
		\centering
  \def\svgwidth{\columnwidth}
        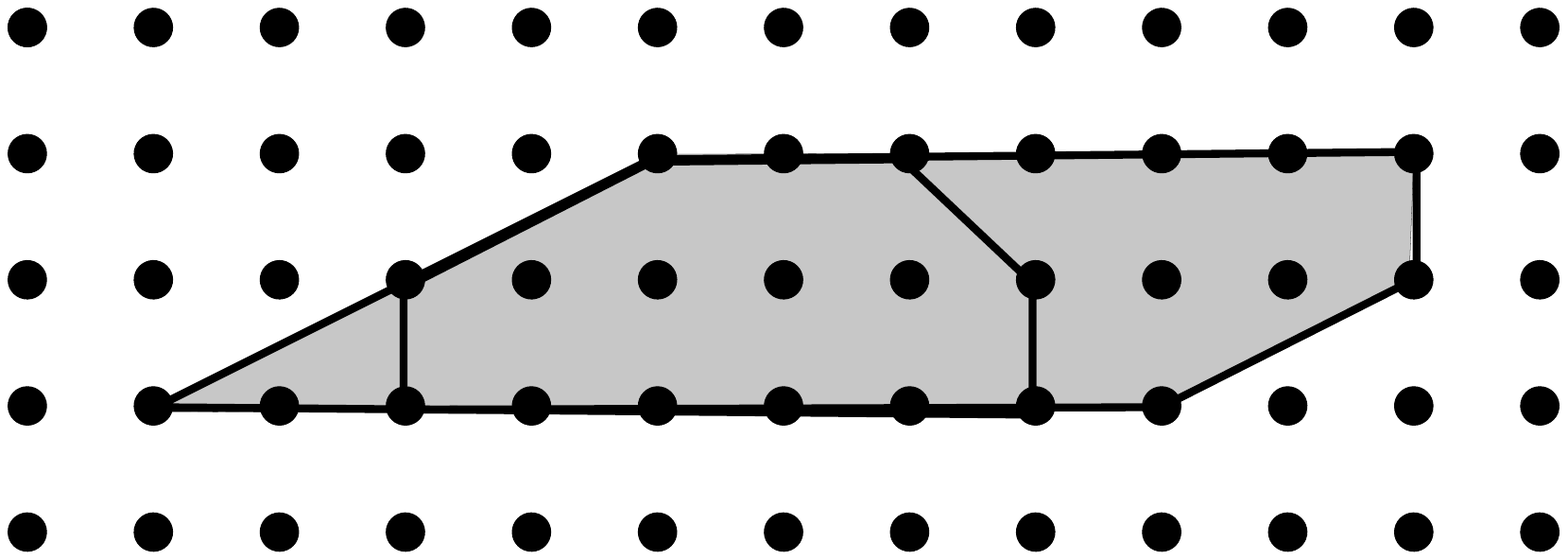
                \setlength{\abovecaptionskip}{-63mm}
		\caption{Rational lines parallel to $\ell$ intersecting the shaded set $\S_0$ contain at least $2$ integer points.}
	\end{subfigure}
		\vspace{-2 in}
	\caption{Steps in the proof of Proposition~\ref{lemma:balanced} when $\ell$ is neither vertical nor horizontal.}
	\label{fig:lemma}
\end{figure}

Therefore, $P_{\eta}(T_a)=P_{\eta}(R_{n,3})$ and we obtain
$$
D_{\eta}(T_a)=D_{\eta}(R_{n,3})-a\leq -a.
$$
Observe that any line parallel to $\ell$ that intersects $\{(0,2),(1,2),\dots,(a-1,2)\}$ must intersect 
$T_a$ in precisely one integer point.  Inductively applying Corollary~\ref{modifiedgeneratingcor}, 
we have that for each $0\leq i<a$, the point $(i,2)$ is not $\eta$-generated by the set $T_a\setminus\{(0,2),\dots,(i-1,2)\}$ and so $D_{\eta}(T_a\setminus\{(0,2),\dots,(i-1,2)\})\leq D_{\eta}(T_a)$.  Setting $\tilde{T}_a:=T_a\setminus\{(0,2),(1,2),\dots,(a-1,2)\}$, it follows that $D_{\eta}(\tilde{T}_a)\leq -a$.
Define
$$
\S_0:=\tilde{T}_a\setminus\{(0,n-a),(0,n-a+1),\dots,(0,n-1)\}.
$$
By Corollary~\ref{amountincrease}, $D_{\eta}(\S_0)\leq
D_{\eta}(T_a)+a\leq0$. (See Figure~\ref{fig:lemma}.)  Moreover, 
every line parallel to $\ell$ that has nonempty intersection with $\S_0$
intersects it in at least two places.

We claim that $\S_0$ contains an $\ell$-balanced subset.  Let $\ell_0$
be the translation of $\ell$ that has nonempty intersection with $w$ and
for $0<i\leq n-1$, let $\ell_i:=\ell_0+(i,0)$.  Then for all $i$,
$\ell_i\cap\S_0\neq\emptyset$ and every element of $\S_0$ is contained
in exactly one of $\ell_0,\dots,\ell_{n-1}$.  Let
$$
U_i:=\bigcup_{j=0}^{n-1}\ell_j\cap\S
$$
and observe that $U_0=\S_0$.  Thus $D_{\eta}(U_0)\leq0$.  If
$D_{\eta}(U_{n-1})\leq0$, then $U_{n-1}$ contains an $\eta$-generating
set.  Since $U_{n-1}$ is a convex subset of a single line, the
Morse-Hedlund Theorem~\cite{MH}
implies that $\eta$ is periodic, a contradiction.  Therefore we have
that $D_{\eta}(U_{n-1})>0$ and there is a maximal index $0\leq
i_{\max}<n-1$ such that $D_{\eta}(U_{i_{\max}})\leq0$.

Write $\ell_{i_{\max}}\cap\S_0=\{q_1,q_2,q_3\}$, where $q_1$ is the
bottom-most element and $q_3$ is the top-most.  If both $q_1$ and $q_3$
are $\eta$-generated by
$U_{i_{\max}}$,
then $U_{i_{\max}}$ is $\ell$-balanced and we are done (here we are using
the fact that every line parallel to $\ell$ that has nonempty intersection
with $\S_0$ intersects it in at least two places).  Otherwise,
without loss of
generality, suppose $q_3$ is not $\eta$-generated by $U_{i_{\max}}$.  Set
$$
\S_1:=U_{i_{\max}}\setminus\{q_3\}.
$$
Since this removes a non-generated
vertex from a set of nonpositive discrepancy, it follows 
that $D_{\eta}(\S_1)\leq D_{\eta}(U_{i_{\max}})\leq0$.  We claim that both $q_1$
and $q_2$ are $\eta$-generated by $\S_1$.  Say, for example, that
$q_2$ is not $\eta$-generated by $\S_1$.  Then
$D_{\eta}(\S_1\setminus\{q_2\})\leq0$ and $q_1$ is $\eta$-generated
by $\S_1\setminus\{q_2\}$, as otherwise
$D_{\eta}(U_{i_{\max}+1})\leq D_{\eta}(\S_1)\leq D_{\eta}(U_{i_{\max}})$
contradicting maximality of $i_{\max}$.
By Corollary~\ref{modifiedgeneratingcor}, this contradicts the fact that
$\ell$ is a nonexpansive direction for $\eta$.  The same argument holds
if $q_1$ is not $\eta$-generated and so we conclude that both $q_1$ and
$q_2$
are $\eta$-generated by $\S_1$.  Therefore $\S_1$ is an
$\ell$-balanced set.
\end{proof}

\begin{definition}
Given a nonexpansive direction $\ell$ and an
$\ell$-balanced set $\S^\ell$,
define the associated {\em border $B_\ell(\S^\ell)$} to be the thinnest strip with
edges parallel to $\ell$
that contains $\S^\ell$.
If $w_\ell\in E(\S^\ell)$ is the edge of $\S^\ell$ that is parallel to
$\ell$,
then $B_\ell(S^\ell\setminus w_\ell)$ denotes the thinnest strip with
edges parallel to $\ell$ that contains $S^\ell\setminus w_\ell$.
\end{definition}

Note that if there exists $n\in\N$ satisfying
$P_{\eta}(R_{n,3})\leq3n$, then
Proposition~\ref{lemma:balanced} guarantees the existence of the set 
$\S^{\ell}$ and the boundary
edge $w_\ell$.

\begin{proposition}\label{prop:period}
Suppose there exists $n\in\N$
such that $P_{\eta}(R_{n,3})\leq3n$, $\ell$ is a nonexpansive
direction for $\eta$, and $H$ is a half-plane whose boundary is parallel to
$\ell$.
Then if $f,g\in X_{\eta}$ are such that $f\neq g$ but
$f\rst{H}=g\rst{H}$, then both $f$ and $g$ are periodic with period
vector parallel to $\ell$.

Furthermore,  if there exists an $\ell$-balanced set $\S^\ell$,  $w_\ell\in
E(\S^\ell)$ is the edge of $\S^\ell$ parallel to $\ell$, and
$B_\ell(S^\ell)$ and $B_\ell(S^\ell\setminus w_\ell)$ are the associated
borders, then for any $\vec u\in\ZZ$:
\begin{enumerate}
\item If the restriction $(T^{\vec
u}f)\rst{B_{\ell}({\S}^{\ell}\setminus w_\ell)}$ does not extend
uniquely to an $\eta$-coloring of $B_{\ell}(\S_{\ell})$, then the period
of $(T^{\vec u}f)\rst{B_{\ell}({\S}^{\ell}\setminus w_\ell)}$ is at most
$|w_{\ell}\cap\ZZ|-1$;\label{first}
\item If the restriction $(T^{\vec
u}f)\rst{B_{\ell}({\S}^{\ell}\setminus w_\ell)}$ extends uniquely to an
$\eta$-coloring of $B_{\ell}(\S^{\ell})$, then the period of $(T^{\vec
u}f)\rst{B_{\ell}({\S}^{\ell}\setminus w_\ell)}$ is at most
$2|w\cap\ZZ|-2$.\label{second}
\end{enumerate}
\end{proposition}

Note that if $\eta$ is aperiodic, then by Proposition~\ref{lemma:balanced}, there exists an $\ell$-balanced set $\S^\ell$.

\begin{proof}
We assume that $\ell$ is a nonexpansive direction and there exists $n\in\N$ 
with $P_\eta(R_{n,3})\leq 3n$.  
Let $\S^\ell$ be an $\ell$-balanced set, 
$w_\ell\in E(\S^\ell)$ be the edge of $\S^\ell$ parallel to $\ell$  and let 
$B_\ell(S^\ell)$ and $B_\ell(S^\ell\setminus w_\ell)$ be the associated
borders.  
By definition, $\S^{\ell}\setminus w_{\ell}$ is 
contained in $B_{\ell}(\S^{\ell}\setminus w_{\ell})$.  Find $A\in 
SL_2(\Z)$ such that $A(\ell)$ points vertically downward and define 
$\tilde{\eta}\colon\ZZ\to\A$ by $\tilde{\eta} = \eta\circ A^{-1}$ and 
$\tilde{\S^{\ell}} = A(\S^{\ell})$.  Observe that $\eta$ is 
aperiodic if and only if $\tilde{\eta}$ is aperiodic, and that 
$\tilde{\S^{\ell}}$ is $A(\ell)$-balanced for $\tilde{\eta}$.

Let $f,g\in X_\eta$ be as in the statement of the proposition.  
 Let 
$\tilde{f}:=f\circ A^{-1}$, $\tilde{g}:=g\circ A^{-1}$, and 
$\tilde{w}_{\ell}:=A(w_{\ell})$.  It suffices 
to show that for any $\vec 
u\in\ZZ$ , 
$\tilde{f}, \tilde{g}$ are periodic and 
that $(T^{\vec u}\tilde{f})\rst{A(B_{\ell}(\S^{\ell}\setminus 
w_\ell))}$ satisfies the claimed bounds on its period.

The proof proceeds in three steps.  First we show that the restriction 
of $f$ to the strip $B_{\ell}(\S^{\ell}\setminus w_{\ell})$ is 
periodic.  Next we use this fact to show that $f$ itself is periodic.  
Finally we use the periodicity of $f$ (with some as yet unknown period) to 
establish the claimed bounds on the period of $(T^{\vec 
u}f)\rst{A(B_{\ell}(\S^{\ell}\setminus w_{\ell}))}$.

\subsubsection*{Step 1: Showing 
$\tilde{f}\rst{B_{\ell}(\S^{\ell}\setminus w_{\ell})}$ is periodic}
For $i\in\Z$, let
$$
H_i:=\{(x,y)\in\ZZ\colon x \geq i\}.
$$
By translating the coordinate system if necessary and using the nonexpansivity of 
$\ell$, we can assume that 
$A(H)=H_0$.  Furthermore, there exists a translation $(i,0)$ such that 
$(T^{-(i,0)}\tilde{f})\rst{H_{0}} = (T^{-(i,0)} \tilde{g})\rst{H_{0}}$, but 
$(T^{-(i,0)}\tilde{f})\rst{H_{-1}}\neq (T^{-(i,0)} \tilde{g})\rst{H_{-1}}$.  Without loss, 
we can assume that $i=0$.  
Set $B:=A(B_{\ell}(\S^{\ell}\setminus w_{\ell}))$ and without loss, assume 
that $B\subseteq H_0$ and $B\not\subseteq H_1$.  Choose minimal $L\in\N$ 
such that
\begin{equation}
\label{def:L}
B=\left\{(x,y)\in\ZZ\colon 0\leq x<L\right\}.
\end{equation}
For $i\in\Z$, set
$$
C_i:=\tilde{\S^{\ell}}+(0,i)\text{ and 
}D_i:=\tilde{\S^{\ell}}\setminus\tilde{w}_{\ell}+(0,i).
$$
We claim that for all $i\in\Z$, the $\tilde{\eta}$-coloring 
$\tilde{f}\rst{D_i}$ does not extend uniquely to an 
$\tilde{\eta}$-coloring of $\tilde{f}\rst{C_i}$.
If not, then $\tilde{f}\rst{B}$ extends uniquely to an 
$\tilde{\eta}$-coloring of $B\cup C_i$ for some $i\in\Z$.  Since any 
translation of an $\ell$-balanced set is also $\ell$-balanced, the 
top-most vertex of the edge of $C_{i+1}$ parallel to $A(\ell)$ is 
$\tilde{\eta}$-generated by $C_{i+1}$.
This is the only element of $C_{i+1}$ that is not contained in $B\cup 
C_i$, and so $\tilde{f}\rst{B}$ extends uniquely to an $\eta$-coloring 
of $B\cup C_i\cup C_{i+1}$.
By induction, $\tilde{f}\rst{B}$ extends uniquely to an $\eta$-coloring 
of $B\cup\bigcup_{j\geq i}C_j$.  The bottom-most vertex of the edge of 
$C_i$ parallel to $A(\ell)$ is also $\eta$-generated by $C_i$,
and so a similar induction argument shows that $\tilde{f}\rst{B}$ 
extends uniquely to an $\tilde{\eta}$-coloring of 
$B\cup\bigcup_{j\in\Z}C_j$.  This contradicts the fact that 
$\tilde{f}\rst{H_0}=\tilde{g}\rst{H_0}$ but 
$\tilde{f}\rst{H_{-1}}\neq\tilde{g}\rst{H_{-1}}$ and so the claim 
follows.  Equivalently, for all $j\in\Z$, the $\tilde{\eta}$-coloring 
$(T^{(0,j)}\tilde{f})\rst{D_0}$ does not extend uniquely to an 
$\tilde{\eta}$-coloring of $C_0$.

By Corollary~\ref{cor:bound}, there are at most 
$|\tilde{w}_{\ell}\cap\tilde{\S^{\ell}}|-1=|w_{\ell}\cap\S^{\ell}|-1$
many colorings of $D_0$ that extend non-uniquely to an 
$\tilde{\eta}$-coloring of $C_0$.  Thus
$$
\left|\left\{(T^{(0,i)}\tilde{f})\rst{D_0}\colon 
i\in\Z\right\}\right|\leq|w_{\ell}\cap\S^{\ell}|-1.
$$
For each integer $0\leq x<L$, where $L$ is defined as in~\eqref{def:L}, let $p_x$ be the bottom-most element of 
$\tilde{\S^{\ell}}\cap\{(x,j)\colon j\in\Z\}$.  Set
$$
V:=\{p_x\colon 0\leq x<L\}\text{ and 
}U:=\bigcup_{y=0}^{|w_{\ell}\cap\S^{\ell}|-2}V+(0,y).
$$
Since $\tilde{\S^{\ell}}$ is $A(\ell)$-balanced, $U\subseteq D_0$. 
\begin{figure}[ht]
      \centering
       \def\svgwidth{\columnwidth}
    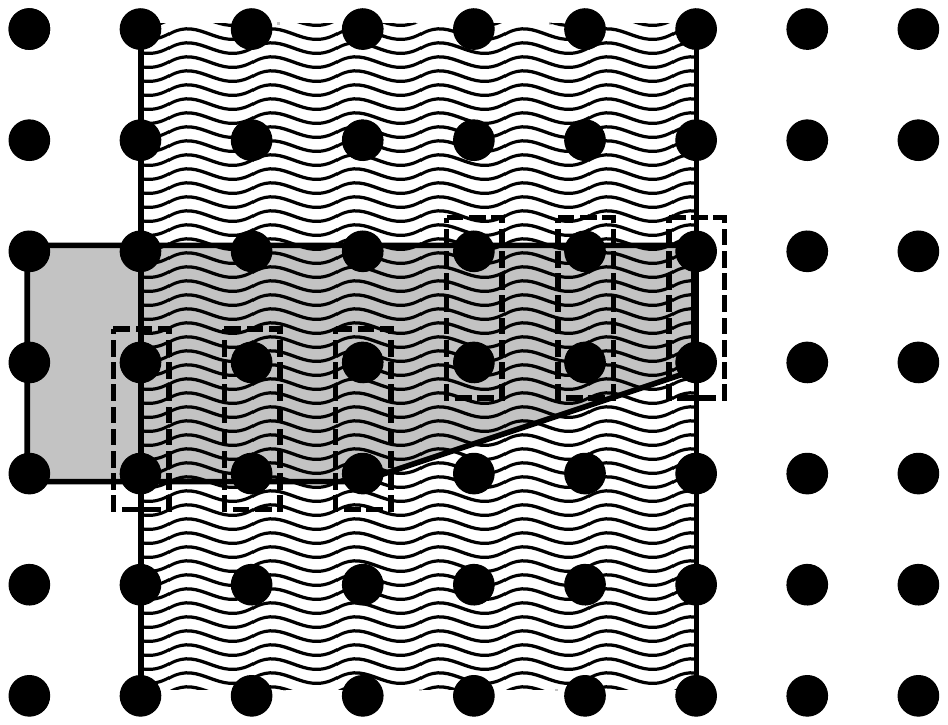
     \setlength{\abovecaptionskip}{-120mm}
      \caption{The shaded region represents $\S$, the union of the boxes is $U$, 
      and the union of the bottom most elements of the boxes is the set $V$.  
      Step $1$ of the proof shows that the wavy region $B$ is periodic.}
     \label{fig:morse-hedlund}
\end{figure}
(See Figure~\ref{fig:morse-hedlund}.)
Define
$\alpha\colon\Z\to\mathcal{W}_{\eta}(V)$
by $\alpha(j):=(T^{(0,j)}\tilde{f})\rst{V}$.  
Patterns of the form 
$\alpha\rst{\{m,m+1,\dots,m+|w_{\ell}\cap\S^{\ell}|-2\}}$ are 
in one-to-one correspondence with patterns of the form 
$(T^{(0,m)}\tilde{f})\rst{U}$.  The number of such patterns is at most 
the number of patterns of the form $(T^{(0,m)}\tilde{f})\rst{D_0}$,
which is at most $|w_{\ell}\cap\S^{\ell}|-1$.  By the 
Morse-Hedlund Theorem~\cite{MH}, $\alpha$ is periodic with period at 
most $|w_{\ell}\cap\S^{\ell}|-1$.  Therefore 
$\tilde{f}\rst{B}$ is vertically periodic with period at most 
$|w_{\ell}\cap\S^{\ell}|-1$ as well.

\subsubsection*{Step 2: Showing $f$ is periodic}  For $i\in\Z$, set
$$
B_i:=B+(i,0).
$$

We claim that for any $i\geq 0$, we have that $\tilde{f}\rst{B_{-i}}$ is 
vertically periodic
and the periods satisfy the bounds in the statement of the proposition.  
For $i=0$, we have already shown that $\tilde{f}\rst{B_0}$ is vertically 
periodic of period at most $|w_{\ell}\cap\S^{\ell}|-1$.  We proceed by induction and 
suppose that for all $0\leq i<k$, we have that $\tilde{f}\rst{B_{-i}}$ is 
periodic and
\begin{enumerate}
\item The period of $\tilde{f}\rst{B_{-i}}$ is at most 
$2|w_{\ell}\cap\S^{\ell}|-2$;
\item If for all $j\in\Z$, the $\eta$-coloring 
$(T^{-(-i,j)}\tilde{f})\rst{\tilde{\S^{\ell}}}$ does not extend uniquely 
to an $\eta$-coloring of $\S^{\ell}$, then the period of 
$\tilde{f}\rst{B_{-i}}$ is at most $|w_{\ell}\cap\S^{\ell}|-1$.
\end{enumerate}

First we show that $\tilde{f}\rst{B_{-k}}$ is vertically periodic of 
period at most $2|w_{\ell}\cap\S^{\ell}|-2$.  Suppose there 
exists $j\in\Z$ such that 
\begin{equation}
\label{eq:extends}
\text{ $(T^{-(-k+1,j)}\tilde{f})\rst{\tilde{\S^{\ell}}}$ extends uniquely to an 
$\eta$-coloring of $\S^{\ell}$.}
\end{equation}
Let 
$p\leq2|w_{\ell}\cap\S^{\ell}|-2$ be the minimal vertical 
period of $\tilde{f}\rst{B_{-k+1}}$.
Then for all $m\in\Z$, 
$(T^{-(-k+1,j+mp)}\tilde{f})\rst{\tilde{\S^{\ell}}}$ extends uniquely to 
an $\eta$-coloring of $\S^{\ell}$ and in particular all the colorings 
$(T^{-(-k+1,j+mp)}\tilde{f})\rst{\tilde{\S^{\ell}}}$ coincide.
By periodicity of $\tilde{f}\rst{B_{-k+1}}$, all of the colorings 
$(T^{-(-k+1,j+mp+1)}\tilde{f})\rst{\tilde{\S^{\ell}}\setminus \tilde{w_\ell}}$ coincide and so 
all of the colorings 
$(T^{-(-k+1,j+mp+1)}\tilde{f})\rst{\tilde{\S^{\ell}}}$ coincide except 
possibly on the top-most element of $\tilde{w_{\ell}}$.
Since $\tilde{\S^{\ell}}$ is $A(\ell)$-balanced, the top-most element of 
$\tilde{w_{\ell}}$ is $\eta$-generated by $\tilde{S^{\ell}}$, and so the colorings 
coincide on the top-most element of $\tilde{w_{\ell}}$ as well.  By 
induction, for any $q$ with $0\leq q<p$ and all $m\in\Z$, all 
colorings of the form $(T^{-(-k+1,j+mp+q)}\tilde{f})\rst{\tilde{\S^{\ell}}}$ 
coincide.  This implies that $\tilde{f}\rst{B_{-k}}$ is periodic and 
that its period divides the period of $\tilde{f}\rst{B_{-k+1}}$.

Otherwise, if~\eqref{eq:extends} does not hold, we can suppose that for all $j\in\Z$, the coloring 
$(T^{-(-i,j)}\tilde{f})\rst{\tilde{\S^{\ell}}\setminus\tilde{w_\ell}}$ does not extend uniquely 
to an $\eta$-coloring of $\tilde{S^{\ell}}$.  Then by applying the 
Morse-Hedlund Theorem as in Step 1, the vertical period of 
$\tilde{f}\rst{B_{-k+1}}$ is at most 
$|w_{\ell}\cap\S^{\ell}|-1$.
As above, let $0\leq p<|w_{\ell}\cap\S^{\ell}|-1$ be the 
minimal vertical period of $\tilde{f}\rst{B_{-k+1}}$.
Let 
$\pi\colon\mathcal{W}_{\eta}(\tilde{\S^{\ell}})\to\mathcal{W}_{\eta}(\tilde{\S^{\ell}}\setminus\tilde{w_\ell})$ 
be the natural restriction map.
As in Corollary~\ref{cor:bound}, there are at most 
$P_{\eta}(\tilde{\S^{\ell}})-P_{\eta}(\tilde{\S^{\ell}}\setminus\tilde{w_\ell})$ elements of 
$\mathcal{W}_{\eta}(\tilde{\S^{\ell}}\setminus\tilde{w_\ell})$ whose pre-image under $\pi$ 
contains more than one element; say the number of such elements is $Q$.
There are at most 
$Q+P_{\eta}(\tilde{S^{\ell}})-P_{\eta}(\tilde{\S^{\ell}}\setminus\tilde{w_\ell})$ elements of 
$\mathcal{W}_{\eta}(\tilde{\S^{\ell}})$ where $\pi$ is not one-to-one.
That is, there are at most
$$2(P_{\eta}(\tilde{\S^{\ell}})-P_{\eta}(\tilde{\S^{\ell}}\setminus\tilde{w_\ell}))\leq2|w_{\ell}\cap\S^{\ell}|-2$$
many $\eta$-colorings of $\tilde{\S^{\ell}}$ whose restrictions to 
$\tilde{\S^{\ell}}\setminus\tilde{w_\ell}$ do not extend uniquely to  $\eta$-colorings of 
$\tilde{\S^{\ell}}$.

Each of the colorings $(T^{-(-k+1,j)}\tilde{f})\rst{\tilde{\S^{\ell}}}$ is such a 
coloring.
By the pigeonhole principle, there exist $0\leq i<j<p$ such that
\begin{equation}
\label{eq:restriction}
(T^{-(-k+1,j_1)}\tilde{f})\rst{\tilde{\S^{\ell}}}=(T^{-(-k+1,j_2)}\tilde{f})\rst{\tilde{\S^{\ell}}}.
\end{equation}
Since $\tilde{\S^{\ell}}$ is $A(\ell)$-balanced, every vertical line with 
nonempty intersection with $\tilde{\S^{\ell}}$ contains at least 
$|w_{\ell}\cap\S^{\ell}|-1$ integer points.  Since the 
vertical period of $\tilde{f}\rst{B_{-k+1}}$ is at most 
$|w_{\ell}\cap\S^{\ell}|-1$ and by using~\eqref{eq:restriction},
we have that $j_2-j_1$ is a multiple of $p$.
Using induction as previously, we have that
$$(T^{-(-k+1,j_1+j)}\tilde{f})\rst{\tilde{\S^{\ell}}}=(T^{-(-k+1,j_2+j)}\tilde{f})\rst{\tilde{\S^{\ell}}}$$
for all $j\in\Z$.  In particular $\tilde{f}\rst{B_{-k+1}\cup B_{-k}}$ is 
vertically periodic of period at most 
$2|w_{\ell}\cap\S^{\ell}|-2$.

By induction, for all $k>0$ we have that $\tilde{f}\rst{B_{-k}}$ is vertically periodic with the 
bounds claimed in the proposition.  Let 
$\T^{\ell}\subseteq R_{n,3}$ be a set with is balanced in the direction 
antiparallel to $\ell$.  Since the restriction of $\tilde{f}$ to the 
vertical half-plane $\{(x,y)\in\ZZ\colon x\leq 0\}$ is periodic, a similar induction argument (using 
$\T^{\ell}$ in place of $\S^{\ell}$) shows that $\tilde{f}$ is 
vertically periodic on all of $\ZZ$, where the precise bounds on the period are yet 
to be determined.  (A priori, these bounds depend on the number of integer points 
on the edge of $\T^\ell$ that is antiparallel to $\ell$.) 

\subsubsection*{Step 3: Showing that the period of $f$ satisfies the 
claimed bounds}  We are left with showing that $\tilde{f}\rst{B_k}$ 
satisfies the claimed bounds for all $k\in\Z$.
We remark that the argument showing that $\tilde{f}\rst{B_{-k}}$ is 
vertically periodic with the claimed bounds relied only the fact that 
$\tilde{f}\rst{B_0}$ was vertically periodic of period at most 
$|w_{\ell}\cap\S^{\ell}|-1$.
Thus it suffices to show that for infinitely many $k>0$, the vertical 
period of $\tilde{f}\rst{B_k}$ is at most 
$|w_{\ell}\cap\S^{\ell}|-1$, since then the previous argument 
shows that the half-plane to the left of such a $B_k$ satisfies the 
claimed bounds.
As before, it further suffices to show that for infinitely many $k>0$, 
the $\eta$-coloring $\tilde{f}\rst{B_k}$ does not extend uniquely to an 
$\eta$-coloring of $B_k\cup B_{k-1}$.

Since $\tilde{f}\rst{B_k}$ is vertically periodic for all $k$, there 
are only finitely many colorings $B_0$ that are of the form 
$(T^{-(k,0)}\tilde{f})\rst{B_0}$ for some $k\in\Z$.
Say there exists an integer $k_{\min} \geq 0$ such that 
$(T^{-(k,0)}\tilde{f})\rst{B_0}$ extends uniquely to an $\eta$-coloring 
of $B_0\cup B_{-1}$ for all $k>k_{\min}$ and without loss assume that 
$k_{\min}$ is the minimal integer with this property.
Let $K\geq k_{\min}$ be the smallest integer for which there exists 
$i\in\N$ such that
$$(T^{-(K+i,0)}\tilde{f})\rst{B_0}=(T^{-(K,0)}\tilde{f})\rst{B_0}$$
($K$ exists by the pigeonhole principle).  Then by definition of 
$k_{\min}$, there is a unique extension of this common coloring of 
$B_0\cup B_{-1}$.  In particular, 
$(T^{-(K-1,0)}\tilde{f})\rst{B_0}=(T^{-(K+i-1,0)}\tilde{f})\rst{B_0}$. If 
$K>0$, this contradicts minimality of $k_{\min}$.  If $K=0$ this 
contradicts the fact that $\tilde{f}\rst{B_0}$ does not extend uniquely 
to an $\eta$-coloring of $B_0\cup B_{-1}$,
which is one of the defining characteristics of $\tilde{f}$.  Either 
case leads to a contradiction,
and so we conclude that no such integer $k_{\min}$ exists.  The bounds on 
$\tilde{f}\rst{B_k}$ claimed in the proposition follow.  

The analogous 
argument applied to $g$ implies the periodicity of $g$.
\end{proof}

\begin{corollary}\label{cor:periodicstripimpliesperiodic}
Assume there exists $n\in\N$ such that $P_{\eta}(R_{n,3})\leq3n$.  Suppose 
$\ell$ is an oriented rational line in $\R^2$, $\widehat{\ell}$ is the anti-parallel line, 
$\S^\ell$ is an $\ell$-balanced set, $\widehat{\S^\ell}$ is an $\widehat{\ell}$-balanced set,
$w_{\ell}\in E(\S)$ is the edge parallel to $\ell$ and 
$B\subset\ZZ$ is the thinnest bi-infinite strip with edges parallel and 
antiparallel to $\ell$ that contains $\S^{\ell}\setminus w_{\ell}$.  
If $\eta\rst{B}$ is periodic, then $\eta$ is periodic with period vector 
parallel to $\ell$.
\end{corollary}
\begin{proof}
Let $\S^\ell$ be an $\ell$-balanced set and let $w_\ell\in E(\S)$ be the associated edge and 
$B$ the associated strip.  
The argument is nearly identical to the proof of Step 2 of 
Proposition~\ref{prop:period} and so we just summarize the differences. 
Maintaining the notation in that proof, if there exists $i\in\Z$ such 
that $\tilde{f}\rst{B_i}$ extends uniquely to an $\eta$-coloring of 
$B_i\cup B_{i-1}$, then $\tilde{f}\rst{B_i}$ is periodic of period at 
most $|w_{\ell}\cap\S^{\ell}|-1$ and the remainder of the 
induction is identical. Otherwise, for every $i\in\Z$, the 
coloring $\tilde{f}\rst{B_i}$ extends uniquely to an $\eta$-coloring of 
$B_i\cup B_{i-1}$.  By the pigeonhole principle and the fact that 
$\S^{\ell}$ is $\ell$-balanced, as in Step 2 of 
Proposition~\ref{prop:period}, it follows that whenever 
$\tilde{f}\rst{B_i}$ is vertically periodic, $\tilde{f}\rst{B_{i-1}}$ is 
vertically periodic of period dividing that of $\tilde{f}\rst{B_i}$.  
This establishes the result for the restriction of $f$ to 
$\bigcup_{j=0}^{\infty}B_{i-j}$.  The restriction to the other 
half-plane follows a similar argument using the antiparallel line $\widehat{\ell}$ 
and associated $\widehat{\ell}$-balanced set $\widehat{\S}^{\ell}$ instead of 
$\S^{\ell}$.
\end{proof}

\begin{corollary}\label{cor:semiinfinite}
Suppose there exists $n\in\N$ such that $P_{\eta}(R_{n,3})\leq3n$ and 
$f\in X_{\eta}$.
Suppose $\ell$ is a nonexpansive direction for $\eta$, $\vec u\in\ZZ$ is 
the shortest integer vector parallel to $\ell$, $\S$ is an 
$\ell$-balanced set, and $w\in E(\S)$ is the edge parallel to $\ell$.  
Let $B_{\ell}(\S\setminus w)$ be the intersection of $\ZZ$ with all 
lines parallel to $\ell$ that have nonempty intersection with 
$\S\setminus w$.  Finally, suppose there exists $R\in\N$ such that for all $r\geq R$, 
$(T^{r\cdot\vec u}f)\rst{\S\setminus w}$ does not extend uniquely to an 
$\eta$-coloring of $\S$.  Then $f\rst{B_\ell(\S\setminus 
w)}$ is eventually periodic with period vector parallel to $\vec u$, 
period at most $|w\cap\S|-1$, and the initial portion which may not be periodic has length at most $|w\cap\S|-1$.
\end{corollary}
\begin{proof}
The proof is almost identical to Step 1 of 
Proposition~\ref{prop:period}.  Define
$$
\alpha\colon\N\to\{(T^{r\cdot\vec u}f)\rst{B_{\ell}(\S\setminus 
w)}\colon r\geq R\}
$$
by setting $\alpha(i):=(T^{r\cdot\vec u}f)\rst{B_{\ell}(\S\setminus 
w)}$.  As in Proposition~\ref{prop:period}, we have that the number of 
patterns of the form 
$\alpha\rst{\{m,m+1,\dots,m+|w\cap\S|-2\}}$ is at most 
$|w\cap\S|-1$.  The one-sided version of the Morse-Hedlund 
Theorem~\cite{MH} shows that $\alpha$ is eventually periodic with
period at most $|w\cap\S|-1$ and is such that the initial portion has length at most $|w\cap\S|-1$.
\end{proof}

\begin{corollary}\label{cor:semiinfiniteperiodbounds}
Assume there exists $n\in\N$ such that $P_{\eta}(R_{n,3})\leq3n$.  
Suppose $\ell$ is an oriented rational line and there exists an 
$\ell$-balanced set $\S^{\ell}$.  Let $w_{\ell}\in E(\S^{\ell})$ be the 
edge parallel to $\ell$ and suppose $\T\subset\ZZ$ is an infinite convex 
set with a semi-infinite edge $W$ parallel to $\ell$.  Let
$$
U:=\left\{\vec u\in\ZZ\colon (\S^{\ell}\setminus w_\ell)+\vec 
u\subseteq\T\text{ and }w_\ell+\vec u\not\subseteq\T\right\}.
$$
If $\eta\rst{(\S\setminus w_{\ell})+U}$ is periodic with period vector 
parallel to $\ell$, then $\eta\rst{\S+U}$ is periodic with period vector 
parallel to $\ell$.  Moreover if for all $\vec u\in U$ the coloring 
$(T^{\vec u}\eta)\rst{\S\setminus w_{\ell}}$ does not extend uniquely to 
an $\eta$-coloring of $\S$, then the period of $\eta\rst{(\S\setminus 
w_{\ell})+U}$ is at most $|w_{\ell}\cap\S^{\ell}|-1$ and the 
period of $\eta\rst{\S+\vec u}$ is at most 
$2|w_{\ell}\cap\S^{\ell}|-2$.  Otherwise the period of 
$\eta\rst{\S+\vec u}$ is equal to the period of $\eta\rst{(\S\setminus 
w_{\ell})+U}$.
\end{corollary}
\begin{proof}
This follows from the Morse-Hedlund Theorem and the pigeonhole 
principle, as in Steps 2 and 3 of Proposition~\ref{prop:period}, and in 
Corollary~\ref{cor:semiinfinite}.
\end{proof}

\begin{proposition}
\label{cor:antiparallel}
Assume $\eta$ is aperiodic and there exists $n\in\N$ such that $P_{\eta}(R_{n,3})\leq3n$.  
If $\ell$ is a nonexpansive direction for $\eta$ and $\S$ is an
$\eta$-generating set, then the direction antiparallel to $\ell$ is also
nonexpansive for $\eta$.  In particular, there is an edge
$\widehat{w}_{\ell}\in E(\S)$ antiparallel to $\ell$.
\end{proposition}
\begin{proof}
We proceed by contradiction.  Suppose $\ell$ is nonexpansive but the
antiparallel direction $\widehat{\ell}$ is expansive for $\eta$.
By Corollary~\ref{cor:rational}, $\ell$ is a rational line.  
Let $f,g\in X_{\eta}$ be as in Proposition~\ref{prop:period}.
For convenience assume that $\ell$ points vertically downward by
composing, if needed, with some $A\in SL_2(\Z)$.  Expansivity of
$\widehat{\ell}$ means there exist $a,b\in\N$ such that every
$\eta$-coloring of $[-a+1,0]\times[-b+1,b-1]$ extends uniquely to an
$\eta$-coloring of the larger set $[-a+1,0]\times[-b+1,b-1]\cup\{(1,0)\}$.
(Otherwise, there are rectangles $Q_R = [-R+1, 0]\times [-R+1, R-1]$ 
for every $R\geq 1$ and there exist functions $f_R, g_R\in X_\eta$ such that 
$f_R\rst{Q_R} = g_R\rst{Q_R}$ and $f_R(1,0) = g_R(1,0)$.  Passing to a 
limit we obtain $f_\infty, g_\infty\in X_\eta$ that agree on a half plane 
but disagree at $(1,0)$, contradicting expansivity.)

Then both $\tilde f = f\circ A^{-1}$ and $\tilde g = g\circ A^{-1}$ are
vertically periodic and agree on a vertical
half plane and so at most one of $\tilde f$ and $\tilde g$ is
horizontally periodic.
Without loss, assume that $\tilde f$ is not horizontally periodic. Let
$C$ be the set of $\tilde f$-colorings of the border
$B_{\ell}(\S^{\ell}\setminus w_\ell)$,
where $\S^\ell$ is an $\ell$-balanced set and $w_\ell$ is the edge of
$\S^\ell$ parallel to $\ell$.
(Note that such a set $\S^\ell$ exists by Proposition~\ref{lemma:balanced}.)  The set 
$C$ is finite because $B_{\ell}({\S}^{\ell}\setminus
w_\ell)$ is a vertical strip and $\tilde f$ is vertically periodic.  We
produce a coloring $\alpha\colon\Z\to C$ by coloring the integer $i$
with the color $(T^{(-i,0)}\tilde f)\rst{B_{\ell}(\S^{\ell}\setminus
w_\ell)}$. Since every $\eta$-coloring of $[-a+1,0]\times[-b+1,b-1]$
extends uniquely to an $\eta$-coloring of
$[-a+1,0]\times[-b+1,b-1]\cup\{(1,0)\}$, we also have that every
$\eta$-coloring of $[-a+1,0]\times(-\infty,\infty)$ extends uniquely to
an $\eta$-coloring of $[-a+1,1]\times(-\infty,\infty)$. Therefore for
any $i\in\Z$, the $\alpha$-color of $\{i,i+1,\dots,i+a-1\}$ uniquely
determines the $\alpha$-color of $i+a$.  Therefore $\alpha$ is periodic
and hence $\tilde f$ is horizontally periodic, a contradiction.  Thus 
$\widehat{\ell}$ is expansive for $\eta$.

By Proposition~\ref{parallelprop}, there is an edge
$\widehat{w}_{\ell}\in E(\S)$ antiparallel to $\ell$.
\end{proof}

\begin{corollary}\label{cor:balanced}
Assume that $\eta$ is aperiodic and there exists $n\in\N$ such that
$P_{\eta}(R_{n,3})\leq3n$.  Let $\S\subseteq R_{n,3}$ be an
$\eta$-generating set satisfying~\eqref{eq:subset}.  Then for
every nonhorizontal, nonexpansive direction $\ell$, $\S$ is $\ell$-balanced.

If $\ell$ is horizontal and nonexpansive, then $\S$ is either $\ell$-balanced or $\widehat{\ell}$-balanced, 
where $\widehat{\ell}$ is the antiparallel direction.
\end{corollary}
\begin{proof}
Assume that $\ell$ is a nonhorizontal and nonexpansive direction.  
We check the four conditions of Definition~\ref{def:balanced}.
The first condition follows from Proposition~\ref{parallelprop}, the
second is immediate from the definition of an $\eta$-generating set and
the third follows since $\S$ satisfies~\eqref{eq:subset}. If
$|w\cap\S|=2$, then the fourth condition follows since every
line with nonempty intersection with $\S$ intersects in at least one
point.  If $|w\cap\S|=3$, then $\ell$ is either vertical or
determines a line with slope of the form $1/a$ for some integer $a>0$.
By Proposition~\ref{cor:antiparallel}, there exists $w_{\widehat{\ell}}\in
E(\S)$ antiparallel to $\ell$.  Since both endpoints of $w_{\widehat{\ell}}$
are boundary vertices of $\S$,
$|w_{\widehat{\ell}}\cap\S|\geq2$.  Therefore any line parallel
to $\ell$ that has nonempty intersection with $\S$, intersects $\S$ in at
least two integer points.

If $\ell$ is horizontal, let $n$ be the smaller of the number of integer points on 
the top and bottom edges of $\S$.  By convexity of $\S$, the middle line 
has length $r\geq n$ for some $r\in \R$.  Thus the middle line contains at least 
$\lfloor r\rfloor \geq n$ integer points, and so $\S$ is balanced for either 
$\ell$ or $\widehat{\ell}$.  
\end{proof}

\section{complexity with multiple nonexpansive lines}
In this section, we show that the complexity assumption of the existence of $n\in\N$ such that $P_{\eta}(R_{n,3})\leq3n$ is incompatible with the existence of more than one nonexpansive line for $\eta$.  

We assume throughout that:
\begin{align}
\label{eq:H1}
\tag{H1}
& X_{\eta} \text{ has at least two nonexpansive lines.}\\
\label{eq:H2}
\tag{H2}
& \text{There exists } n\in\N \text{ such that } P_{\eta}(n,3)\leq3n.
\end{align}

If $\eta$ is periodic, let $\vec u\in\ZZ$ be a period vector and consider any line $\ell$ that is not parallel to $\vec u$.  By taking a neighborhood of $\ell$ wide enough to include $\ell\pm\vec u$, 
we have that $\ell$ is expansive.  Thus every line apart from possibly the direction determined by $\vec u$ is expansive, so there is at most one nonexpansive line. 
Thus Hypothesis~\eqref{eq:H1} implies that 
\begin{equation}
\label{eq:aperiodic}
\eta \text{ is aperiodic.}
\end{equation}

We begin with some general facts about the shape of an $\eta$-generating set.  
By Proposition~\ref{parallelprop}, if $\S$ is an $\eta$-generating set,
then the boundary $\partial\S$ contains an edge parallel to each
nonexpansive direction.  By Proposition~\ref{cor:antiparallel}, whenever
$\ell$ is a nonexpansive direction, the direction antiparallel to $\ell$
is also a nonexpansive direction.  Since $\S\subseteq R_{n,3}$,
$\partial\S$ cannot consist of more than six edges (at most two edges are 
horizontal and the others connect integer points in $R_{n,3}$ with 
different $y$-coordinates).  Thus there are at most
three nonexpansive lines for $\eta$, and each orientation on each line
determines a nonexpansive direction. Furthermore, by 
Proposition~\ref{parallelprop}, we can assume that all of the 
nonexpansive lines are rational lines through the origin.

We start with a construction of  a large convex set that 
is used in Propositions~\ref{exactlytwo} and~\ref{exactlythree} to show that $\eta$ cannot have multiple nonexpansive lines while also having low complexity.

As noted, we have at most three nonexpansive lines for $\eta$.  
Let 
\begin{equation}
\label{eq:lines}
\ell_1, \ell_2\subset\R^2 \text{ or } \ell_1, \ell_2, \ell_3\subset\R^2 \text{ denote 
the nonexpansive lines for } \eta,
\end{equation}
depending if there are $2$ or $3$ nonexpansive 
lines. We write all statements for three nonexpansive lines, with the implicit 
understanding that when there are only $2$ nonexpansive lines, we remove any reference 
to $\ell_3$.  

Without loss of generality, we can assume that all $\ell_i$ pass through the origin.  
By Corollary~\ref{cor:rational}, we can assume that the nonexpansive lines are rational lines and without loss we can assume that $\ell_1, \ell_3$ are not horizontal.  
By Proposition~\ref{cor:linedirection}, there exist orientations on $\ell_1, \ell_2, \ell_3$ that determine nonexpansive directions for $\eta$.  
For the remainder of this construction, we make a slight abuse of notation and view $\ell_1, \ell_2, \ell_3$ as directed lines that determine nonexpansive directions.

Let $\S\subseteq R_{n,3}$ be an $\eta$-generating set.  By Proposition~\ref{prop:generating-sets}, there exist edges $w_1, w_2, w_3\in E(\S)$ parallel to $\ell_1, \ell_2, \ell_3$, respectively.  By Proposition~\ref{cor:antiparallel}, there exist $\widehat{w}_1, \widehat{w}_2, \widehat{w}_3\in E(\S)$ such that $w_i$ is antiparallel to $\widehat{w}_i$, for $i=1,2,3$.  By Corollary~\ref{cor:balanced}, since $w_1$ and $w_3$ are not 
horizontal, we have that $\S$ is $w_1, \widehat{w}_1, w_3$ and $\widehat{w}_3$-balanced.  If $w_2$ is not horizontal, then again applying  Corollary~\ref{cor:balanced}, we have that $\S$ is both $w_2$ and $\widehat{w}_2$-balanced.  If $w_2$ is horizontal, then $\S$ is balanced for 
at least one of $w_2$ and $\widehat{w}_2$.  So, without loss, we can assume that
$$
\text{$\S$ is $w_1, \widehat{w}_1,w_3, \widehat{w}_3$ and $w_2$-balanced.}
$$

Let $H^{\prime}_0$ denote the half-plane through the origin determined by $\ell_1$.  Let $H^{\prime}_{-1}$ be the smallest half-plane strictly containing $H^{\prime}_0$ whose boundary contains an integer point (this is well-defined since $\ell_1$ is a rational line).  Since $\ell_1$ is a nonexpansive direction, there exist $f,g\in X_{\eta}$ such that $f\rst{H^{\prime}_0}=g\rst{H^{\prime}_0}$ but $f\rst{H^{\prime}_{-1}}\neq g\rst{H^{\prime}_{-1}}$.  
Since $\ell_2$ is not parallel to $\ell_1$ and $f\rst{H^{\prime}_0}=g\rst{H^{\prime}_0}$, 
at most one of $f\rst{H^{\prime}_{-1}}$ and $g\rst{H^{\prime}_{-1}}$ extends to a $\ZZ$-coloring that is periodic with period vector parallel to $\ell_2$.  
Without loss of generality, suppose $f\rst{H^{\prime}_{-1}}$ is an $\eta$-coloring of $H^{\prime}_{-1}$ which cannot be extended to a periodic $\eta$-coloring of $\ZZ$ with a period vector parallel to $\ell_2$.  By Proposition~\ref{prop:period}, $f$ is periodic with period vector parallel to $\ell_1$.  Translating if needed, we can assume that $(w_1\cap\ZZ)\subset H^\prime_{-1}\setminus H^\prime_0$.  It follows that $\S\subset H^\prime_{-1}$ (recall that the boundaries of both $\S$ and $H^\prime_{-1}$ are positively oriented).

To make the constructions clearer, it is convenient to make a change of coordinates such that $\ell_1$ points vertically downward.  
Thus choose $A\in SL_2(\Z)$ such that $A(\ell_1)$ points vertically downward.  Define 
\begin{equation}
\label{eq:tildes}
\tilde{\eta}:=\eta\circ A^{-1}; 
\tilde{f}:=f\circ A^{-1};
\tilde{\S}:=A(\S),
\end{equation}
and 
\begin{equation}
\label{eq:moretilde}
\tilde{\ell}_i:=A(\ell_i) \text{, } \tilde{w}_i:=A(w_i) \text{ and }\widehat{\tilde{w}_i}:=A(\widehat{w}_i), \text{ for } i=1,2,3.  
\end{equation}
 Then for any finite, nonempty set $\T\subset\ZZ$, we have $D_{\eta}(\T)=D_{\tilde{\eta}}(A(\T))$.  It follows that $\tilde{\eta}$ is aperiodic and 
 \begin{equation}
 \label{eq:tildef-periodic}
 \tilde{f} \text{ is vertically periodic  with minimal period } p \text{ and is not doubly periodic.}
 \end{equation}
 Further, 
 \begin{equation}
 \label{eq:Stilde-generating}
 \tilde{\S} \text{ is an } \tilde{\eta}\text{-generating set }
 \end{equation}
 and
 \begin{equation}
 \label{eq:Stilde-balanced}
 \tilde{\S} \text{ is } \tilde{w}_1, \widehat{\tilde{w}_1}, \tilde{w}_3, \widehat{\tilde{w}_3}\text{-balanced and is balanced for at least one of } \tilde{w}_2 \text{ and } \widehat{\tilde{w}_2}.
 \end{equation}

For $i\in\Z$, define
$$
H_i:=\left\{(x,y)\in\ZZ\colon x\geq i\right\}.
$$
Note that $H_0=A(H^{\prime}_0)$ and $H_{-1}=A(H^{\prime}_{-1})$.  
For $i\in\Z$, let $B_i$ be a vertical strip of width $i$ defined by 
\begin{equation}
\label{def:Bi}
B_i:=H_{-1}\setminus H_{i-2}
\end{equation}
and $\bar{B}_i$ be the vertical sub-strip of width $i-1$ defined by 
$$\bar{B}_i:=H_0\setminus H_{i-2}.$$ 
Let $d\in\N$ be the number of distinct vertical lines passing through $\tilde{\S}$ and note that $\tilde{\S}\subset B_d$ and $(\tilde{\S}\setminus w_1)\subset\bar{B}_d$.  

We claim there are infinitely many integers $x\geq 0$ such that
\begin{equation}
\label{eq:non-unique}
\tilde{f}\rst{\bar{B}_d+(x,0)} \text{ does not extend uniquely to an } \eta\text{-coloring of } B_d+(x,0).
\end{equation}  
By construction, $x=0$ is such an integer.  
If there are not infinitely many such integers, let $x_{\max}$ denote the largest such integer.  
By~\eqref{eq:tildef-periodic}, $\tilde{f}$ is vertically periodic and there are only finitely many colorings of the form $(T^{(x,0)}\tilde{f})\rst{\bar{B}_d}$; say there are $P$ such colorings.  
By the pigeonhole principle, there are distinct integers $x_1, x_2\in\{x_{\max}+1,\dots,x_{\max}+P+2\}$ such that 
$$(T^{(x_1,0)}\tilde{f})\rst{\bar{B}_d}=(T^{(x_2,0)}\tilde{f})\rst{\bar{B}_d};$$ 
without loss assume that $x_1\geq x_{\max}$ is the smallest integer for which there exists $x_2$ with this property.  Since $(T^{(x_2,0)}\tilde{f})\rst{\bar{B}_d}$ extends uniquely to an $\eta$-coloring of $B_d$, so does $(T^{(x_1,0)}\tilde{f})\rst{\bar{B}_d}$.  Therefore 
$$(T^{(x_1-1,0)}\tilde{f})\rst{\bar{B}_d}=(T^{(x_2-1,0)}\tilde{f})\rst{\bar{B}_d}.$$  
Since $x_2-1>x_{\max}$, we have that $(T^{(x_2-1,0)}\tilde{f})\rst{\bar{B}_d}$ extends 
uniquely to an $\eta$-coloring of $B_d$.  Thus so does $(T^{(x_1-1,0)}\tilde{f})\rst{\bar{B}_d}$, 
and in particular, we have that $x_1-1>x_{\max}$.  However, this contradicts the choice of $x_1$ 
as the smallest integer with this property and the claim follows.  

Let $0=x_1<x_2<x_3<\ldots$ be a sequence integers satisfying~\eqref{eq:non-unique}.  
Then since $\tilde{\S}$ is $\tilde{w}_1$-balanced by~\eqref{eq:Stilde-balanced} and 
for all $i\in\N$, $f\rst{A^{-1}(\bar{B}_d+(x_i,0))}$ satisfies condition~\eqref{first} in Proposition~\ref{prop:period} and so it has period at most $|w_1\cap\S|-1=|\tilde{w}_1\cap\tilde{\S}|-1$.  
It follows that for all $i\in\N$, $\tilde{f}\rst{\bar{B}_d+(x_i,0)}$ is vertically periodic of period at most $|\tilde{w}_1\cap\tilde{\S}|-1$.

\begin{claim}
\label{claim:infinite}
For all $i\geq d$, there is no finite set $F_i\subset B_i$ such that every $\tilde{\eta}$-coloring of the form $(T^{(0,j)}\tilde{f})\rst{F_i}$ extends uniquely to an $\tilde{\eta}$-coloring of $B_i$. 
\end{claim} 
If not, suppose $F_i\subset B_i$ is a finite set and for all $j\in\Z$ the coloring $(T^{(0,j)}\tilde{f})\rst{F_i}$ extends uniquely to an $\eta$-coloring of $B_i$.  Since $\tilde{f}\in X_{\tilde{\eta}}$, 
there exists $\vec u\in\ZZ$ such that 
$$
\tilde{f}\rst{F_i}=(T^{\vec u}\tilde{\eta})\rst{F_i}, 
$$
where the existence of $\vec u$ follows from the fact that every finite pattern occurring in an element of $X_{\tilde{\eta}}$ also occurs in $\tilde{\eta}$.  Therefore $(T^{\vec u}\tilde{\eta})\rst{B_i}=\tilde{f}\rst{B_i}$ is vertically periodic.  By Corollary~\ref{cor:periodicstripimpliesperiodic}, we have that $\tilde{\eta}$ is periodic and thus that $\eta$ is periodic, a contradiction of~\eqref{eq:aperiodic}.  The claim follows.

To describe the large set we construct, we define: 
\begin{definition}
If $\S\subseteq\ZZ$ is a convex set, then $\T\subseteq\ZZ$ is {\em $E(\S)$-enveloped} if
\begin{enumerate}
\item $\T$ is convex;
\item For all $w\in E(\T)$, there exists $u\in E(\S)$ such that $w$ is parallel to $u$ and $|w|\geq|u|$.
\end{enumerate}
\end{definition}

Maintaining notation of $\tilde{f}$ and $\tilde{S}$ defined in~\eqref{eq:tildes} and $B_1$ defined in~\eqref{def:Bi}, we inductively define a convex set $G_\infty$ on which we can control periodicity.  For each $i\in\N$, let 
\begin{equation}\label{it:one}
\begin{tabular}{l}
$F_i\subset B_{d+i-1}$ be a finite, $E(\tilde{\S})$-enveloped set \\
containing $[-1,i-1]\times[-d-i-1,d+i+1]$.
\end{tabular}
\end{equation}
and let
\begin{equation}\label{it:two}
\begin{tabular}{l}
$G_i\subseteq B_{d+i-1}$ be the largest $E(\tilde{\S})$-enveloped \\
set to which $\tilde{f}\rst{F_i}$ extends uniquely
\end{tabular}
\end{equation}
(we allow the possibilities that $G_i=F_i$ or that $G_i$ is infinite). 

By Claim~\ref{claim:infinite},  
$G_j\neq B_j$ and so the set 
\begin{equation}
\label{eq:G-infty}
G_j\cap\{(-1,y)\colon y\in\Z\}
\end{equation}
is semi-infinite.  This semi-infinite line either has an element of maximal $y$-coordinate or of minimal $y$-coordinate.  Therefore there is either a subsequence $\{j_k\}_{k=0}^{\infty}$ such that $G_{j_k}\cap\{(-1,y)\colon y\in\Z\}$ has an element of maximal $y$-coordinate for all $k$ or there is a subsequence such that $G_{j_k}\cap\{(-1,y)\colon y\in\Z\}$ has an element of minimal $y$-coordinate for all $k$.  Without loss of generality (the other case being similar), suppose that there are infinitely many $j\in\N$ such that the set $G_j\cap\{(-1,y)\colon y\in\Z\}$ has an element of maximal $y$-coordinate.  Without loss (passing to a subsequence if necessary) we assume $G_j \cap\{(-1,y)\colon y\in\Z\}$ has an element of maximal $y$-coordinate for all $j\in\N$ and let $y_j^{\max}$ be this $y$-coordinate.  By~\eqref{eq:tildef-periodic}, $\tilde{f}$ is vertically periodic with minimal period $p$.  There exists $0\leq J_{\max}<p$ such that for infinitely many $j$, $y_j^{\max}\equiv J_{\max}$ (mod $p$).  Passing to this subsequence and maintaining the same notation on indices $j$, for each such $j$, let $k_j\in\Z$ be such that $y_j^{\max}=k_j\cdot p+J_{\max}$.  By periodicity, $\tilde{f}\rst{G_j-(0,k_j\cdot p)}$ does not extend uniquely to an $\eta$-coloring of any larger convex set and the point $(-1,J_{\max})$ is the top-most element of $\{(-1,y)\colon y\in\Z\}\cap(G_j-(0,k_j\cdot p))$ for all $j$.

Set 
\begin{equation}
\label{def:G-inf}
G_{\infty}:=\bigcup_j\left(G_j-(0,k_j\cdot p)\right).
\end{equation}
If necessary, we again make a change of 
coordinates and assume that $J_{\max}=0$.  Thus
\begin{equation}
\label{eq:G-infty-one}
\text{$G_{\infty}$ is an $E(\tilde{\S})$-enveloped set that intersects every vertical line in $H_{-1}$.}
\end{equation}
By construction, $E(G_{\infty})$ has a semi-infinite edge that points vertically downward from $(0,0)$.  By~\eqref{eq:G-infty-one},
\begin{equation}
\label{def:u}
G_\infty \text{ has a nonhorizontal, semi-infinite edge } u\in E(G_{\infty})
\end{equation}
and  $u$ is parallel to some edge in $E(\tilde\S)$.
This edge determines a nonexpansive direction for $\tilde{\eta}$, since $\tilde{f}\rst{G_{\infty}}$ cannot be uniquely extended to any larger $E(\tilde\S)$-enveloped set.

Define $K\supset G_{\infty}$ by taking 
\begin{equation}
\label{def:K}
K \text{ is the smallest } E(\tilde{\S})\text{-enveloped set containing } G_{\infty} \text{ with }
u\notin \partial K, 
\end{equation}
meaning that $K$ is the set obtained by extending the successor edge to $u$ backwards until it intersects an integer point and then taking the convex hull (note that successor 
edge is meant with respect to positive orientation on the boundary).  By construction, 
\begin{equation}
\label{eq:h-exists}
\text{ there exists } \tilde{h}\in X_{\tilde{\eta}} \text{ such that } \tilde{f}\rst{G_{\infty}}=\tilde{h}\rst{G_{\infty}} \text{ and } \tilde{f}\rst{K}\neq\tilde{h}\rst{K}. 
\end{equation}
By~\eqref{eq:tildef-periodic}, $\tilde{f}$ is vertically periodic and so 
\begin{equation}
\label{eq:h-periodic}
\tilde{h}\rst{G_\infty} \text{ is vertically periodic (with minimal period p) but $\tilde{h}\rst{K}$ is not.}
\end{equation}

We use the construction of $G_\infty$ to eliminate the case of $2$ nonexpansive lines: 
\begin{proposition}
\label{exactlytwo}
Suppose there are exactly two nonexpansive lines for $X_{\eta}$.  Then for all $n\in\N$, $P_{\eta}(R_{n,3})>3n$.
\end{proposition}
\begin{proof}
We proceed by contradiction and assume that $\eta$ has exactly two 
nonexpansive directions and that there exists $n\in\N$ such that
$P_{\eta}(R_{n,3})\leq3n$.  Thus hypotheses~\eqref{eq:H1} and~\eqref{eq:H2} are satisfied.
In particular, by~\eqref{eq:aperiodic}, $\eta$ is aperiodic.  

We maintain the notation of the nonexpansive lines in~\eqref{eq:lines} (where we assume only two), 
the quantities in~\eqref{eq:tildes} and~\eqref{eq:moretilde},  and of the construction of the set $G_\infty$ defined in~\eqref{def:G-inf} satisfying~\eqref{eq:G-infty-one}.  
Since there are only two nonexpansive lines for $\eta$,  the edge $u$ defined in~\eqref{def:u} must either be parallel or antiparallel to $\tilde{\ell}_2$.  Let $K\supset G_{\infty}$ be defined as in~\eqref{def:K} and $\tilde{h}$ as in~\eqref{eq:h-exists}. 
Then
 $K\setminus G_{\infty}$ can be written as
$$
K\setminus G_{\infty}=\bigcup_{k=1}^{k_{0}}(l_k\cap K), 
$$
where $l_1, l_2,\dots,l_{k_{0}}$ are (undirected) lines parallel to $\tilde{\ell}_2$ and $k_0$ is the number of lines produced in the construction of $K$.  By~\eqref{eq:h-periodic}, $\tilde{h}\rst{K}$ cannot be extended to a vertically periodic $\eta$-coloring of $H_{-1}$.  Let $u_0:=u$ and label the edges of $G_{\infty}$ by $u_{i+1}:=\suc(u_i)$ for $i=0,\dots,|E(G_{\infty})|-1$, 
where $\suc(\cdot)$ denotes the successor edge taken with positive orientation.  

Suppose $u_I$ is the edge parallel to $\tilde{\ell}_1$, meaning that $u_I$ points vertically downward.  Define a sequence of sets
$$
G_{\infty}=L_0\subset L_1\subset L_2\subset\cdots\subset L_{I-1}, 
$$
where $L_{i+1}$ is obtained from $L_i$ by extending the edge of $L_i$ parallel to $u_{I-i}$ to be semi-infinite and taking the intersection of $\ZZ$ with the convex hull of the resulting shape (see Figure~\ref{fig:construction}).  Then $E(L_{i+1})=E(L_i)\setminus\{u_{I-i+1}\}$.

\begin{figure}[ht]
      \centering
   \def\svgwidth{\columnwidth}
    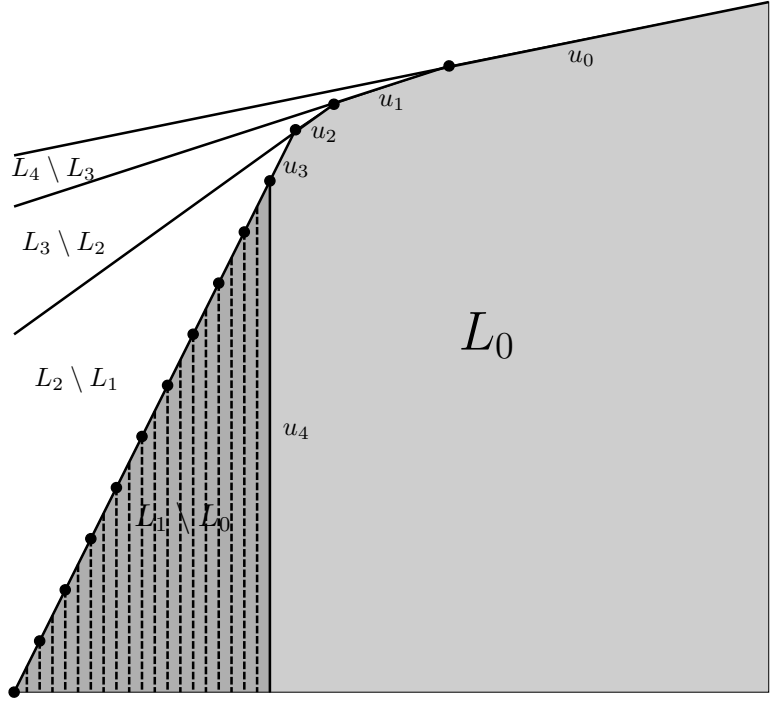
         \setlength{\abovecaptionskip}{-80mm}
      \caption{The construction of the nested sets $L_0\subset 
L_1\cdots$.  Integer
      points at the intersection of two lines are marked with a dot and 
the dotted
      lines show $L_1\setminus L_0=s_{1}\cup s_{2}\cup\cdots$.}
      \label{fig:construction}
\end{figure}

We claim that for $0\leq i<I$, $\tilde{h}\rst{L_i}$ is vertically periodic, but possibly of larger period than that of $\tilde{h}\rst{L_0}$ and that $\tilde{h}\rst{L_I}$ is eventually vertically periodic.   For $i=0$, this follows directly from the construction of $G_\infty$.  
For $i=1$, write
$$
L_1\setminus L_0=s_{1}\cup s_{2}\cup\cdots
$$
where $s_j$ is the semi-infinite line defined by $s_j:=\{(-j-1,y)\colon y\in\Z\}\cap L_1$.  
For integers $0\leq a\leq b$, write $s_{[a,b]} = s_a\cup s_{a+1}\cup\ldots\cup s_b$.
Suppose that $\tilde{h}\rst{L_0\cup s_{[1,j]}}$ is vertically periodic.  
Let $\vec v_j(i)\in\ZZ$ be the translation of $\tilde{\S}$ such that the top-most element of $\tilde{w}_1+\vec v_j(i)$ is the point $(j,i)$.  
If for all $R<0$ 
there exists $r\leq R$ such that $(T^{-\vec v_j(r)}\tilde{h})\rst{\tilde{\S}\setminus\tilde{w}_1}$ extends uniquely to an $\eta$-coloring of $\tilde{\S}$, then there is a unique extension of $\tilde{h}\rst{L_0\cup s_{[1, j]}}$ to an $\tilde{\eta}$-coloring of $L_0\cup s_{[1, j+1]}$ by~\eqref{eq:Stilde-generating}.  
In this case, arguing as in Step $1$ in Proposition~\ref{prop:period}, 
the restriction of $\tilde{h}$ to $L_0\cup s_{[1, j+1]}$ is vertically periodic of the same period as $\tilde{h}\rst{L_0\cup s_{[1,j]}}$.  
Otherwise there exists $R<0$ such that for all $r\leq R$ the coloring $(T^{-\vec v_j(r)}\tilde{h})\rst{\tilde{\S}\setminus\tilde{w}_1}$ does not extend uniquely to an $\eta$-coloring of $\tilde{\S}$.  
Then by Corollary~\ref{cor:semiinfinite} the restriction of $\tilde{h}$ to $L_{1}\cap s_{[j, j+m]}$ is eventually periodic of period at most $|\tilde{w}_1\cap\tilde{\S}|-1$ and the initial portion which may not be periodic has length at most $|\tilde{w}_1\cap\tilde{\S}|-1$, where $m$ is the number of vertical lines in $L_1$ which have nonempty intersection with $\tilde{\S}\setminus\tilde{w}_1$.  
So by Corollary~\ref{cor:semiinfiniteperiodbounds}, we have that $\tilde{h}\rst{L_{1}\cap s_{j+1}}$ is eventually vertically periodic of period at most $2|\tilde{w}_1\cap\tilde{\S}|-2$ and the initial portion which may not be periodic again has length at most $|\tilde{w}_1\cap\tilde{\S}|-1$.  
It follows by induction that $\tilde{h}\rst{L_1}$ is eventually vertically periodic and $(T^{(0,|\tilde{w}\cap\tilde{\S}|+1)}\tilde{h})\rst{L_1}$ is vertically periodic.  If $L_1=L_I$ then the claim follows.  Otherwise the semi-infinite edge of $L_1$ parallel to $u_{I-1}$ determines an expansive direction for $\tilde{\eta}$.  Write 
$$
L_2\setminus\bigl(L_0\cup(L_1-(0,|\tilde{w}\cap\tilde{\S}|-1))\bigr)=\bigcup_{i=1}^{k_1}\tilde{s}_i
$$
where the $\tilde{s}_i$ are semi-infinite lines parallel to $u_{I-1}$.  Since $u_{I-1}$ is expansive, there is a unique extension of $\tilde{h}\rst{L_1-(0,|\tilde{w}\cap\tilde{\S}|-1)}$ to an $\eta$-coloring of $(L_1-(0,|\tilde{w}\cap\tilde{\S}|-1))\cup\tilde{s}_1$.  
Since $L_1-(0,|\tilde{w}\cap\tilde{\S}|-1)$ is colored in the same way as $L_1-(0,|\tilde{w}\cap\tilde{\S}|-1-q)$, where $q$ is the vertical period, and there is a unique way to extend this coloring to an $\tilde{\eta}$-coloring of $(L_1-(0,|\tilde{w}\cap\tilde{\S}|-1))\cup\tilde{s}_1$, we have that
the vertical periodicity of $\tilde{h}\rst{L_1-(0,|\tilde{w}\cap\tilde{\S}|-1)}$ implies that $\tilde{h}\rst{\bigl(L_1-(0,|\tilde{w}\cap\tilde{\S}|-1)\bigr)\cup\tilde{s}_1}$ is also vertically periodic. 
Inductively it follows that $\tilde{h}\rst{L_2}$ is vertically periodic.
More generally, suppose that $\tilde{h}\rst{L_i}$ is vertically periodic for $i<I$.  Then $L_i$ has two semi-infinite edges, one of which it shares with $L_0$ and the other determines an expansive direction for $\tilde{\eta}$.  Write
$$
L_{i+1}\setminus L_i=t_1\cup t_2\cup\cdots
$$
where $L_i\cup t_{[1,j]}$ is convex for all $j=1,2,\dots$,  each $t_j$ is the intersection of $\ZZ$ with a semi-infinite line parallel to $u_{I-i}$ and contained in $L_i$, and 
$t_{[a,b]} = t_a\cup t_{a+1}\cup\ldots \cup t_b$.  
Suppose that $\tilde{h}\rst{L_i\cup s_{[1,j]}}$ is vertically periodic.  Since $u_{I-i}$ determines an $\tilde{\eta}$-expansive direction, 
there is a unique extension of $L_i\cup t_{[1,j]}$ to an $\tilde{\eta}$-coloring of $L_i\cup t_{[1,j+1]}$.  
By vertical periodicity, $\tilde{h}\rst{L_i\cup t_{[1,j]}}=(T^{(0,-q)}\tilde{h})\rst{L_i\cup t_{[1,j]}}$, 
where $q$ denotes the smallest vertical period of $\tilde{h}\rst{L_i\cup t_{[1,j]}}$.  
By uniqueness, $\tilde{h}\rst{L_i\cup t_{[1,j+1]}}=(T^{(0,-q)}\tilde{h})\rst{L_i\cup t_{[1,j+1]}}$ and hence is also vertically periodic.  By induction, this holds for all $j$ and hence $\tilde{h}\rst{L_{i+1}}$ is vertically periodic.  The claim follows.

Let $C$ denote the smallest bi-infinite strip whose edges are parallel to $\tilde{\ell}_2$ that contains 
$\tilde{\S}\setminus\tilde{w}_2$.  Let $J\in\Z$ be the maximal integer such that $C+(0,J)$ is a subset of the region in $\ZZ$ on which $\tilde{h}$ is vertically periodic, let $C_j:=C+(0,j)$, and let $Q\in\N$ be the smallest vertical period of $\tilde{h}\rst{L_I-(0,J)}$.  The integer $J$ is well-defined by~\eqref{eq:h-periodic}.  
Then for all $j\leq J$, we have that $\tilde{h}\rst{C_j}=(T^{(0,-Q)}\tilde{h})\rst{C_j}$.  

We claim that for all $j\leq J$, $\tilde{h}\rst{C_j}$ is not periodic with period vector parallel to $\tilde{\ell}_2$.  
By the preceding remark, it suffices to show that this holds for all sufficiently negative values of $j$.
For all $j\in\Z$ sufficiently negative that the only edge of $L_0$ that $C_j$ intersects is the edge parallel to $\tilde{\ell}_1$ (all but finitely many $C_j$ have this property), recall that $\tilde{h}\rst{L_0}=\tilde{f}\rst{L_0}$.  
By the construction of $\tilde{f}$, we have that $\tilde{f}\rst{H_{-1}}$ cannot be extended to an $\tilde{\eta}$-coloring of $\ZZ$ which is periodic with period vector parallel to $\tilde{\ell}_2$.  
If $\tilde{h}\rst{C_j}$ is $\tilde{\ell}_2$-periodic, then by Corollary~\ref{cor:periodicstripimpliesperiodic} it follows that $\tilde{h}$ itself is $\tilde{\ell}_2$-periodic.  
But the sequence $(T^{(0,-k)}\tilde{h})$ has an accumulation point, and any such accumulation 
point is also $\tilde{\ell}_2$-periodic.  
Moreover, the restriction of any such accumulation point to $H_{-1}$ is one of the functions $\tilde{f}\rst{H_{-1}}, (T^{(0,-1)}\tilde{f})\rst{H_{-1}},\dots,(T^{(0,-p+1)}\tilde{f})\rst{H_{-1}}$ (where again $p\in\N$ is the minimal vertical period of $\tilde{f}$).  This contradicts the fact that $\tilde{f}\rst{H_{-1}}$ does not extend to a $\tilde{\ell}_2$-periodic coloring of $\ZZ$, and the claim follows.

If $\ell_2$ is not horizontal, then $\tilde{\S}$ is $u$-balanced, where $u$ is the edge defined in~\eqref{def:u}.    In this case every line parallel to $u$ that has nonempty intersection with $\tilde{\S}$ contains at least $|\tilde{w}_2\cap\tilde{\S}|-1$ integer points.  Since $\tilde{h}\rst{C_j}$ is not $\tilde{\ell}$-periodic, the Morse-Hedlund Theorem implies that there are at least $|\tilde{w}_2\cap\tilde{\S}|$ distinct $\tilde{\eta}$-colorings of $\tilde{\S}\setminus\tilde{w}_2$ that occur in $C_j$ (otherwise the coloring would be periodic).  But there are at most $|\tilde{w}\cap\tilde{\S}|-1$ $\eta$-colorings of $\tilde{\S}\setminus w_2$ that extend non-uniquely to an $\eta$-coloring of $\tilde{\S}$, and so by Corollary~\ref{eq:tildef-periodic}, the coloring of $C_j$ extends uniquely to an $\tilde{\eta}$-coloring of $C_j\cup C_{j+1}$ for all $j\leq J$.  
Since the restriction of $\tilde{h}$ to the region $\bigcup_{j\leq J}C_j$ is vertically periodic and $\tilde{h}\rst{C_J}$ extends uniquely to an $\tilde{\eta}$-coloring of $C_J\cup C_{J+1}$, the restriction of $\tilde{h}$ to the region $\bigcup_{j\leq J+1}C_j$ is vertically periodic.  But this contradicts the definition of $J$.  If $\ell_2$ is horizontal, then the same argument applies to $\S^{\ell_2}$ in place of $\S$, where $\S^{\ell_2}$ is an $\ell_2$-balanced subset of $R_{n,3}$ constructed by Proposition~\ref{lemma:balanced}.
\end{proof}

Following standard terminology in the literature (e.g.~\cite{EKM}) we make the following definition:
\begin{definition}
\label{def:periodic-in-region}
Suppose $\T\subset\ZZ$ and $\vec u\in\ZZ$.  We say that $\alpha\colon 
\T\to\A$ is {\em periodic when restricted to  the region $T$ with period 
vector $\vec u$} if $\alpha(\vec x)=\alpha(\vec x+\vec u)$ for all $\vec 
x\in\T$ such that $\vec x+\vec u\in\T$.
\end{definition}

\begin{proposition}\label{exactlythree}
Suppose there are exactly three nonexpansive lines for $\eta$.  Then for all $n\in\N$, $P_{\eta}(R_{n,3})>3n$.
\end{proposition}
\begin{proof}
We proceed by contradiction and assume that $\eta$ has exactly three 
nonexpansive directions and that there exists $n\in\N$ such that
$P_{\eta}(R_{n,3})\leq3n$.  Thus hypotheses~\eqref{eq:H1} and~\eqref{eq:H2} are satisfied.  In particular, $\eta$ is aperiodic~\eqref{eq:aperiodic}.

By Proposition~\ref{parallelprop}, there exists an $\eta$-generating set $\S\subseteq R_{n,3}$ which satisfies~\eqref{eq:subset} and every nonexpansive direction for $\eta$ is parallel to one of the edges of $\S$.  By Proposition~\ref{cor:antiparallel}, the direction antiparallel to any nonexpansive direction is also nonexpansive.  Since there are exactly three nonexpansive lines for $\eta$, $\S$ has precisely six edges, all of which determine nonexpansive directions.  
Since $\S\subseteq R_{n,3}$, two of these edges must be horizontal and the remaining four edges each contain exactly two integer points.  
Again by Proposition~\ref{cor:antiparallel}, every edge of $\S$ is antiparallel to another edge of $\S$, and so $\partial\S$ is a hexagon comprised of three pairs of parallel edges.  It follows that the two horizontal edges contain the same number of integer points and this number is at most $n-1$.  
Let $w_1\in E(\S)$ be the predecessor edge to the top horizontal edge in $E(\S)$ and recursively define $w_{i+1}:=\suc(w_i)$ for $i=1,2,3,4,5$ (see Figure~\ref{fig:three-nonexp}).  
Then $w_{i+3}$ is antiparallel to $w_i$ for all $i$, where the indices are understood to be taken $\pmod 6$.

\begin{figure}[ht]
      \centering
   \def\svgwidth{.8\columnwidth}
         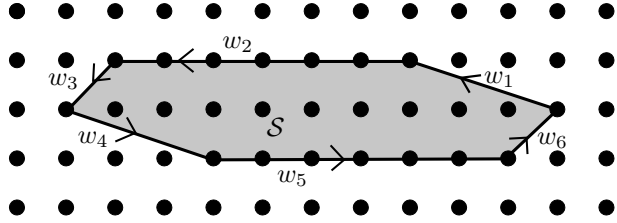
         \setlength{\abovecaptionskip}{-100mm}
      \caption{The set $\S$ with oriented edges labeled.}
      \label{fig:three-nonexp}
\end{figure}

We summarize: $|w_2\cap\ZZ|=|w_5\cap\ZZ|\leq n-1$ and $|w_i\cap\ZZ|=|w_{i+3}\cap\ZZ|=2$ for $i=1,3$.  It follows that 
\begin{equation}\label{eq:fact}
\text{$\S$ is balanced in every nonexpansive direction.}
\end{equation}

For convenience, define $a_1, a_3, a_4, a_6\in\Z$ such that
\begin{equation*}
\text{$w_i$ is parallel to $(a_i,1)$ for $i=1,6$ and $w_i$ is parallel to $(a_i,-1)$ for $i=3,4$.}
\end{equation*}
By convexity, one of the statements: 
\begin{align*}
& a_1, a_3\leq0; \\
& a_1\leq0, a_3\geq0, |a_1|>a_3; \\
& a_1\geq0, a_3\leq0, |a_3|>a_1; 
\end{align*}
holds.  
In each case, every horizontal line that has nonempty intersection with $\S$ contains at least 
\begin{equation}\label{eq:size-of-S*}
|w_2\cap\S|\text{ integer points}
\end{equation}
(e.g. in the first case the middle horizontal line in $\S$ contains $|w_2\cap\S|+|a_1|+|a_3|$ integer points, and the other cases are similar).\footnote{This bound is stronger than our usual bound that every horizontal line that has nonempty intersection with $\S$ intersects in at least $|w_2\cap\S|-1$ integer points.}

For $j\in\Z$, let $V_j$ be the horizontal half-plane defined by 
$$
V_j:=\left\{(x,y)\colon x\in\Z, \ y\leq j\right\}.
$$
Since the $w_2$ direction is nonexpansive for $\eta$, by Proposition~\ref{prop:period} there exist $f,g\in X_{\eta}$ such that $f\rst{V_0}=g\rst{V_0}$ but $f\rst{V_1}\neq g\rst{V_1}$.  At most one of $f$ and $g$ is periodic with period vector parallel to $w_1$, and so we can 
suppose without loss that $f$ is not.  Furthermore, without loss we can assume that 
\begin{equation}\label{eq:f-not-periodic}
f\rst{V_1} \text{ does not extend to a periodic } \eta\text{-coloring of } \ZZ 
\text{ with period parallel to } w_1.
\end{equation}

Since $\S$ is $w_2$-balanced by~\eqref{eq:fact}, it follows from Proposition~\ref{prop:period} that
$f$ is horizontally periodic and the restriction of $f$ to any horizontal strip of height two has period at most $2|w_2\cap\S|-2$.  Set
$$
B:=\left\{(x,y)\in\ZZ\colon y\in\{-1,0\}\right\}\text{ and }C:=\left\{(x,y)\colon y\in\{-1,0,1\}\right\}.
$$
For any $j\in\Z$ such that $(T^{-(0,j)}f)\rst{B}$ does not extend uniquely to an $\eta$-coloring of $C$, we have that $(T^{-(0,j)}f)\rst{B}$ is horizontally periodic of period at most $|w_1\cap\S|-1$.  In particular, this holds for $j=0$.  

We claim that there are infinitely many integers $j\leq0$ 
such that 
\begin{equation}
\label{eq:inf-non-unique}
(T^{-(0,j)}f)\rst{B} \text{ does not extend uniquely to an } \eta\text{-coloring of } C.
\end{equation}
The proof of the claim is similar to that of~\eqref{eq:non-unique}.  
We proceed by contradiction.  Suppose that there exists an integer $J\leq0$ such that for all $j<J$, the coloring $(T^{-(0,j)}f)\rst{B}$ extends uniquely to an $\eta$-coloring of $C$ and assume that $|J|$ is minimal. 
Since $f\rst{V_0}$ is horizontally periodic, there are only finitely many $\eta$-colorings of the form $(T^{-(0,j)}f)\rst{B}$ for $j\leq0$.   Say there are $M$ such colorings.  
 Then by the pigeonhole principle, there exist $1\leq j_1<j_2\leq M+2$ such that $(T^{-(0,J-j_1)}f)\rst{B}=(T^{-(0,J-j_2)}f)\rst{B}$.  Choose $j_1$ to be the smallest integer such that there exists $j_2$ with this property.  Then by construction, $(T^{-(0,J-j_1)}f)\rst{C}=(T^{-(0,J-j_2)}f)\rst{C}$ and hence $(T^{-(0,J-j_1+1)}f)\rst{B}=(T^{-(0,J-j_2+1)}f)\rst{B}$.  If $j_1>1$, this contradicts the minimality of $j_1$.  If $j_1=1$, then the fact that $(T^{-(0,J-j_2+1)}f)\rst{B}=(T^{-(0,J)}f)\rst{B}$ extends uniquely to an $\eta$-coloring of $C$ contradicts the minimality of $|J|$.  The claim follows.

Let
\begin{align}
\S_R&:=\S\text{ with the rightmost element of every row removed;}\label{eq:S_R} \\
\S_L&:=\S\text{ with the leftmost element of every row removed.}\label{eq:S_L}
\end{align}

We claim that there do not exist integers $y_1, y_2\in\Z$ such that both of the following hold simultaneously:
\begin{align}
\label{eq:def-S_R} 
\text{for all $x\in\Z$, $(T^{(x,y_1)}f)\rst{\S_R}$ extends uniquely to an $\eta$-coloring of $\S$;} \\
\label{eq:def-S_L} 
\text{for all $x\in\Z$, $(T^{(x,y_2)}f)\rst{\S_L}$ extends uniquely to an $\eta$-coloring of $\S$.}
\end{align}
 We prove the claim by contradiction.  Suppose instead that such integers $y_1, y_2\in\Z$ exist and assume $y_1\leq y_2$ (the other case being similar).  Define $F:=[0,|\S|]\times[y_1,y_2+2]$ and observe that since $f\in X_{\eta}$,  there exists $\vec u\in\ZZ$ such that $f\rst{F}=(T^{\vec u}\eta)\rst{F}$.  By~\eqref{eq:def-S_R} and~\eqref{eq:def-S_L}, $T^{\vec u}\eta$ coincides with $f$ on the set
\begin{equation}
\label{eq:big-set}
F\cup\left([0,\infty)\times[y_1,y_1+2]\right)\cup\left((-\infty,0]\times[y_2,y_2+2]\right),
\end{equation}
and so $T^{\vec u}\eta$ is horizontally periodic on this set.  
Let $v\in V(\S)$ be the vertex of intersection of the edges $w_1$ and $w_2$.  There is a translation of $\S$ that takes $v$ to the point $(|\S|+1,y_1+3)$ and takes $\S\setminus v$ to a subset of $F\cup([0,\infty)\times[y_1,y_1+2])$.  Since $\S$ is $\eta$-generating and $T^{\vec u}\eta$ coincides with $f$ on $F\cup([0,\infty)\times[y_1,y_1+2])$, we have that
$$(T^{\vec u}(\eta))(|\S|+1,y_1+3)=f(|\S|+1,y_1+3).
$$  
It follows by induction that $(T^{\vec u}(\eta))(|\S|+k,y_1+3)=f(|\S|+k,y_1+3)$ for all $k\geq1$.  A similar induction argument shows that 
$$(T^{\vec u}(\eta))(|\S|+k,y_1+2+j)=f(|\S|+k,y_1+2+j)$$ 
for all $k\geq1$
and all $1\leq k\leq y_2-y_1$.  Therefore $T^{\vec u}\eta$ and $f$ coincide on the set larger than 
in~\eqref{eq:big-set}, defined by:
\begin{equation*}
F\cup\left([0,\infty)\times[y_1,y_2+2]\right)\cup\left((-\infty,0]\times[y_2,y_2+2]\right).
\end{equation*}
A similar argument, using the vertex $v^{\prime}\in V(\S)$ that is the intersection of 
the edges $w_4$ and $w_5$ in place of $v$, shows that $T^{\vec u}\eta$ and $f$ coincide on the set
$$
(-\infty,\infty)\times[y_1,y_2+2],
$$
and so $T^{\vec u}\eta$ is horizontally periodic on this set.  Since $\S$ is horizontally balanced by~\eqref{eq:fact} it follows from Corollary~\ref{cor:periodicstripimpliesperiodic} that $T^{\vec u}\eta$ is horizontally periodic and hence $\eta$ is periodic.  This is a contradiction of~\eqref{eq:aperiodic} and the claim follows.  

Thus henceforth we assume that for all $y\in\Z$, there exists $x_y\in\Z$ such that
\begin{equation}\label{eq:extension}
(T^{(x_y,y)}f)\rst{\S_R} \text{ does not extend uniquely to an } \eta\text{-coloring of } \S.
\end{equation}
(The remainder of the proof is analogous if instead, for all $y\in\Z$, there exists $x_y\in\Z$ such that $(T^{(x_y,y)}f)\rst{\S_L}$ does not extend uniquely to an $\eta$-coloring of $\S$.)

\begin{claim}
\label{claim:done}
There exists a nonpositive integer $y$ such that $f\rst{V_y}$ is doubly periodic, $f\rst{V_{y+1}}$ is not doubly periodic, and either $(-a_1,-1)$ or $(-a_6,-1)$ is a period vector for $f\rst{V_y}$.  
\end{claim}
As $V_y$ is a half plane, doubly periodic is  interpreted in the sense of Definition~\ref{def:periodic-in-region}.  Recall that $B=\left\{(x,y)\in\ZZ\colon y\in\{-1,0\}\right\}$.  Let $B^{\prime}$ be the thinnest strip with edges parallel and antiparallel to $w_1$ which contains $\S\setminus w_1$.  For $x\in\Z$, let 
\begin{equation*}
B_x^{\prime}:=B^{\prime}+(x,0).
\end{equation*}
If there exists $x_0\in\Z$ such that $f\rst{B^{\prime}_{x_0}\cap V_0}$ does not extend uniquely to an $\eta$-coloring of $(B^{\prime}_{x_0}\cup B^{\prime}_{x_0+1})\cap V_0$, then for any $\vec u\in\ZZ$ such that $(\S\setminus w_1+\vec u)\subset B^{\prime}_{x_0}\cap V_0$, 
since $\S$ is $\eta$-generating we have that 
$(T^{-\vec u}f)\rst{\S\setminus w_1}$ extends non-uniquely to an $\eta$-coloring of $\S$.  
Since $\S$ satisfies~\eqref{eq:subset}, by Corollary~\ref{generatingequation} we have that $D_{\eta}(\S\setminus w_1)>D_{\eta}(\S)$.  Since $|w_1\cap\S|=2$, there is precisely one coloring of $\S\setminus w_1$ that extends non-uniquely to an $\eta$-coloring of $\S$.  
In particular, since
$$
B^{\prime}_{x_0}\cap V_0=\bigcup_{k=2}^{\infty}\left((\S\setminus w_1)+(x_0-ka_1,-k)\right)
$$
it follows that $B^{\prime}_{x_0}\cap V_0$ is periodic with period vector $(-a_1,-1)$.  
Since $f\rst{V_1}\neq g\rst{V_1}$, we have that $f\rst{B}$ is horizontally periodic of period at most $|w_2\cap\S|-1$.  
The region $(B^{\prime}_{x_0}\cap V_0)\cap B$ is convex and both $\{(x,-1)\colon x\in\Z\}$ and $\{(x,0)\colon x\in\Z\}$ intersect it in at least $|w_2\cap\S|-1$ integer points by~\eqref{eq:size-of-S*}, 
as the strip $B^\prime_{x_0}$ is only wide enough to contain 
$\S\setminus w_1$.  Therefore $f(x,0)=f(x-a_1,-1)$ for all $x\in\Z$.  

Recall that $\S\subseteq R_{n,3}$, and we assume that the bottom most row of $R_{n,3}$ 
lies on the $x$-axis.  
We have that $\S+(x_0-2a_1,-2)$ is contained in the set $(B^{\prime}_{x_0}\cap V_0)\cup B$.  If $v\in V(\S)$ is the vertex incident to $w_6$ and $w_1$, observe that $(\S\setminus v)+(x_0-3a_1,-3)$ is also contained in $(B^{\prime}_{x_0}\cap V_0)\cup B$ and moreover that
$$(T^{-(x_0-2a_1,-2)}f)\rst{\S\setminus v}=(T^{-(x_0-3a_1,-3)}f)\rst{\S\setminus v}.$$  
Since $v$ is $\eta$-generated by $\S$, 
$$(T^{-(x_0-2a_1,-2)}f)\rst{\S}=(T^{-(x_0-3a_1,-3)}f)\rst{\S}.
$$  
It follows by induction that the coloring $f\rst{(B^{\prime}_{x_0}\cup B^{\prime}_{x_0+1})\cap V_0}$ is periodic with period vector $(-a_1,-1)$.  Inductively it follows that the restriction of $f$ to $V_0\cap\bigcup_{k=0}^{\infty}B^{\prime}_{x_0+k}$ is periodic with period vector $(-a_1,-1)$ as well.  
A final induction, where the vertex $v$ is replaced by the vertex $v^{\prime}$ incident to $w_4$ and $w_5$,  shows that $f\rst{V_0}$ is doubly periodic and that $(-a_1,-1)$ is a period vector.  A similar argument applies if there exists $x_0\in\Z$ such that $f\rst{B_{x_0}\cap V_0}$ does not extend uniquely to an $\eta$-coloring of $(B^{\prime}_{x_0-1}\cup B^{\prime}_{x_0})\cap V_0$.  Thus 
we are finished unless for every $x\in\Z$ the coloring $f\rst{B^{\prime}_x\cap V_0}$ extends uniquely to an $\eta$-coloring of $(B^{\prime}_{x-1}\cup B^{\prime}_x\cup B^{\prime}_{x+1})\cap V_0$.  If
$$
D(r):=\bigcup_{k=2}^r(\S\setminus w_1)+(-ka_1,-k), 
$$
it follows that for all $x\in\Z$ there exists $R_x\in\N$ such that for all $r>R_x$, the coloring $(T^{(x,0)}f)\rst{D(r)}$ extends uniquely to an $\eta$-coloring of
$$
\bar{D}(r):=\bigcup_{k=2}^r\S+(-ka_1,-k)\cup \bigcup_{k=2}^r\S+(-ka_1-1,-k)
$$
(we assume $R_x$ is minimal with this property).  Since $f\rst{V_1}$ is horizontally periodic, the set $\{R_x\colon x\in\Z\}$ is finite so $R:=\max_{x\in\Z}R_x$ is well-defined.  It also follows that $f\rst{V_0\setminus V_{-R}}$ is doubly periodic where $(-a_1,-1)$ is one period vector and the horizontal period is at most $|w\cap\S|-1$.   For $s\in\N$, set
$$
E(s):=\bigcup_{k=0}^s(\S\setminus w_6)+(-Ra_1-ka_3,-R-k).
$$
As above, if there exists $x\in\Z$ such that $(T^{(x,0)}f)\rst{E(s)}$ extends non-uniquely to an $\eta$-coloring of 
$$
\bar{E}(s):=\bigcup_{k=0}^s\S+(-Ra_1-ka_3,-R-k)\cup\bigcup_{k=0}^s\S+(-Ra_1-ka_3-1,-R-k),
$$
then $f\rst{V_{R-1}}$ is doubly periodic and $(-a_6,-1)$ is a period for it.  Otherwise, for all $x\in\Z$, there exists $R^{\prime}_x\in\N$ such that for all $r>R^{\prime}_x$, the coloring $(T^{(x,0)}f)\rst{E(s)}$ extends uniquely to an $\eta$-coloring of $\bar{E}(s)$ and again $R^{\prime}:=\max_{x\in\Z}R^{\prime}_x$ is well-defined.  The claim has been shown unless this last case occurs.  In that case, for all $x\in\Z$ the coloring $(T^{(x,0)}f)\rst{D(R)\cup E(R^{\prime})}$ extends uniquely to an $\eta$-coloring of $\bar{D}(R)\cup\bar{E}(R^{\prime})$.  It follows that for all $x\in\Z$, the coloring $(T^{(x,0)}f)\rst{D(R)\cup E(R^{\prime})}$ extends uniquely to an $\eta$-coloring of $V_0\setminus V_{R+R^{\prime}}$.  Since $f\in X_{\eta}$, there exists $\vec u\in\ZZ$ such that $(T^{\vec u}\eta)\rst{D(R)\cup E(R^{\prime})}$ and therefore $f\rst{V_0\setminus V_{R+R^{\prime}}}=(T^{\vec u}\eta)\rst{R+R^{\prime}}$ is horizontally periodic.  By Corollary~\ref{cor:periodicstripimpliesperiodic}, $\eta$ itself is horizontally periodic; a contradiction of~\eqref{eq:aperiodic}.  Therefore either $f\rst{V_0}$ is doubly periodic with period vector $(-a_1,-1)$ or there exists $R\in\N$ such that $f\rst{V_R}$ is doubly periodic with period vector $(-a_6,-1)$.  Claim~\ref{claim:done} follows.

Thus we can define $y_0\leq 0$ to be the integer of least absolute value for which Claim~\ref{claim:done} holds.  Recalling that $f\rst{V_1}$ is not doubly periodic, we have shown:  
\begin{equation}\label{eq:period-of-f-on-region}
\text{ $f\rst{V_{y_0}}$ is doubly periodic 
and  $f\rst{V_{y_0+1}}$ is not doubly periodic, }
\end{equation} 
and either $(-a_1,-1)$ or $(-a_6,-1)$ is a period vector for $f\rst{V_{y_0}}$.  
Henceforth we assume that $i\in\{1,6\}$ is chosen such that $(-a_i,-1)$ is a period vector for $f\rst{V_{y_0}}$.

By~\eqref{eq:inf-non-unique}, 
there exists $j<y_0$ such that $(T^{(0,j)}f)\rst{B}$ is horizontally periodic of period at most $|w_2\cap\S|-1$.  Since $(-a_i,-1)$ is a period vector for $f\rst{V_{y_0}}$, it follows that
the horizontal period of $f\rst{V_{y_0}}$ is at most $|w_2\cap\S|-1$.
By~\eqref{eq:f-not-periodic}, $f\rst{V_1}$ cannot be extended to a periodic coloring of $\ZZ$ with period vector parallel to $w_i$.  It follows that $f\rst{V_{y_0+1}}$ is not doubly periodic (if $y_0<0$ this follows from the definition of $y_0$ and if $y_0=0$ from~\eqref{eq:f-not-periodic}).  Let $p_1\in V(\S)$ be the vertex at the intersection of the edges $w_1$ and $w_2$ and let $p_2\in V(\S)$ be the vertex at the intersection of the edges $w_2$ and $w_3$.  Since $\S$ is $\eta$-generating, if there exists $i\in\{1,2\}$ and $x\in\Z$ such that $(T^{-(x,y_011)}f)\rst{\S\setminus p_i}$ coincides with $(T^{-(x-a_i,y_0-2)}f)\rst{\S\setminus p_i}$, then $f\rst{V_{y_0+1}}$ is doubly periodic, a contradiction.  It follows that 
for all $m\in\Z$, there exists $x\in\{m,m+1,\dots,m+|w_2\cap\S|-2\}$ such that 
\begin{equation}\label{eq:density}
f(x,y_0+1)\neq f(x-a_i,y_0).
\end{equation}

Let $\S_R$ be as in~\eqref{eq:S_R}.  By~\eqref{eq:size-of-S*}, every horizontal line that has nonempty intersection with $\S$ intersects in at least $|w_1\cap\S|$ integer points,  and so every such line intersects $\S_R$ in at least $|w_1\cap\S|-1$ integer points.

We claim that there are at least three distinct $\eta$-colorings of $\S_R$ which extend non-uniquely to an $\eta$-coloring of $\S$.

First by~\eqref{eq:extension}, there exists $x\in\Z$ such that $(T^{-(x,y_0-2)}f)\rst{\S_R}$ does not extend uniquely to an $\eta$-coloring of $\S$ and by~\eqref{eq:period-of-f-on-region} this coloring of $\S_R$ is periodic with period vector $(-a_i,-1)$.  Thus there is an $\eta$-coloring of $\S_R$ that does not extend uniquely to an $\eta$-coloring of $\S$ and this coloring is periodic with period vector $(-a_i,-1)$.

Second, consider the set of colorings of $\S_R$ of the form $(T^{-(x,y_0-1)}f)\rst{\S_R}$.  By~\eqref{eq:extension}, there exists $x_{y_0-1}\in\Z$ such that $(T^{-(x_{y_0-1},y_0-1)}f)\rst{\S_R}$ does not extend uniquely to an $\eta$-coloring of $\S$.  
By~\eqref{eq:density}, there exists a integer point $(x,2)\in w_2$ such that $(T^{-(x_{y_0-1},y_0-1)}f)\rst{\S_R}(x,2)\neq(T^{-(x_{y_0-1},y_0-1)}f)\rst{\S_R}(x-a_i,1)$ but the bottom two horizontal lines of $\S$ are periodic with period vector $(-a_i,-1)$ by~\eqref{eq:period-of-f-on-region}.  Therefore this coloring is distinct from the first coloring of $\S_R$. 

Third, consider the set of colorings of $\S_R$ of the form $(T^{-(x,y_0)}f)\rst{\S_R}$.  Again by~\eqref{eq:extension}, there exists $x_{y_0}\in\Z$ such that $(T^{-(x_{y_0},y_0)}f)\rst{\S_R}$ does not extend uniquely to an $\eta$-coloring of $\S$.  By~\eqref{eq:density}, there exists an integer point $(x,0)\in w_5$ such that $(T^{-(x_{y_0},y_0)}f)\rst{\S_R}(x,0)\neq(T^{-(x_{y_0},y_0)}f)\rst{\S_R}(x+a_i,1)$.  Therefore this coloring is distinct from the first two colorings.  Thus we have three distinct $\eta$-colorings of $\S_R$ which extend non-uniquely to an $\eta$-coloring of $\S$.

But since $\S$ satisfies~\eqref{eq:subset}, we have $D_{\eta}(\S_R)>D_{\eta}(\S)$.  By definition, $|\S_R|=|\S|-3$, and so we have $P_{\eta}(\S)\leq P_{\eta}(\S_R)+2$.  Therefore there are at most two colorings of $\S_R$ that extend non-uniquely to an $\eta$-coloring of $\S$, a contradiction.
\end{proof}

\section{Completing the proof of the main theorem} 

We recall the statement of Theorem~\ref{th:main}:
\begin{theorem*}
Suppose $\eta\colon \ZZ\to\A$ and there exists $n\in\N$ such that $P_{\eta}(n,3)\leq 3n$.  Then $\eta$ is periodic.
\end{theorem*}

\begin{proof}
Suppose there exists $n\in\N$ such that $P_{\eta}(n,3)\leq3n$.  By Proposition~\ref{parallelprop} there exists an $\eta$-generating set $\S\subseteq R_{n,3}$.  Since $\S$ is convex and the endpoints of any edge of $\partial\S$ are integer points in $R_{n,3}$, $E(\S)$ has at most six edges.  Also by Proposition~\ref{parallelprop} every nonexpansive direction is parallel to an edge in $E(\S)$, and so there are at most six nonexpansive directions for $\eta$.  By Proposition~\ref{cor:linedirection}, every nonexpansive line has an orientation that determines a nonexpansive direction.  By Proposition~\ref{cor:antiparallel}, the direction antiparallel to any nonexpansive direction is also nonexpansive (i.e. if $\ell$ is a nonexpansive line then both orientations on $\ell$ determine nonexpansive directions).  Therefore there are at most three nonexpansive lines for $\eta$.

There are four cases to consider.  If there are no nonexpansive lines for $\eta$, then $\eta$ is doubly periodic by Theorem~\ref{thm:doublyperiodic}.  If there is exactly one nonexpansive line for $\eta$, then $\eta$ is singly (but not doubly) periodic by Theorem~\ref{thm:singlyperiodic}.  If there are exactly two nonexpansive lines for $\eta$, then Proposition~\ref{exactlytwo} implies that
$P_{\eta}(R_{n,3})>3n$, a contradiction.  If there are exactly three nonexpansive lines for $\eta$, 
then Proposition~\ref{exactlythree} implies that $P_{\eta}(R_{n,3})>3n$, again a contradiction.  The theorem follows. 
\end{proof}

\section{Further Directions}
Sander and Tijdeman~\cite{ST2} conjectured that for $\eta\colon\ZZ\to\A$, 
if there exists a finite set $\S\subset\ZZ$ 
that is the restriction of a convex set in $\R^2$ to $\ZZ$ and such that $P_\eta(\S)\leq |\S|$, 
where $|\S|$ denotes the number of integer points in $\S$, then $\eta$ is periodic.  
Their result in~\cite{ST} shows that this conjecture holds for rectangles $R_{n,2}$ of height $2$.  
More generally, rephrasing their arguments in our language, 
their proof also covers more 
convex shapes of height $2$.  Namely, if $\S\subset\ZZ$ is a finite set 
that is the restriction of a convex set in $\R^2$ to $\ZZ$ satisfying 
$P_\eta(\S)\leq |\S|$ and such that $\S$ is contained in the union of two adjacent parallel rational 
lines, then $\eta$ is periodic.  
The construction of a generating set works in the more general setting of such a shape $\S$, 
and results in a generating set with $3$ or $4$ edges, and with the possible exception of a single direction (the analog of horizontal) it is balanced.  There can be at most $2$ nonexpansive directions, and we eliminate the case of $2$ in a similar manner to that done for rectangular shapes. 

However, in height $3$, we are unable to generalize our result of Theorem~\ref{th:main} 
to prove the analog for more general convex shapes with a restriction on the height, 
meaning a convex shape contained in a strip of width $3$.  
While the construction of generating sets passes through, resulting in generating sets 
with at most $6$ edges, we are not able to show that they are balanced in all (but perhaps 
the analog of the horizontal) directions.  This is the only hurdle remaining for completing 
a more general result for configurations of height $3$.  

For more general rectangles $R_{n,k}$ with $k\geq 4$, the construction of generating sets, once again, is general. 
Again, a problem arises with proving the existence of balanced sets.  Furthermore, 
the counting of configurations seems to be significantly more difficult in the absence of  
the simple geometry available in height $3$.

\end{document}